\documentclass[11pt, reqno, a4]{amsart}
\usepackage{fullpage}

\usepackage[textwidth=20mm]{todonotes} 
\usepackage{amsmath}
\usepackage{mathrsfs}    
\usepackage{amssymb}
\usepackage{amsthm}
\usepackage{graphicx,color}
\usepackage{hyperref}
\usepackage{bm}

\usepackage[applemac]{inputenc}

\newtheorem{theorem}{Theorem}[section]
\newtheorem*{theorem*}{Theorem}
\newtheorem{proposition}[theorem]{Proposition}

\newtheorem{lemma}[theorem]{Lemma}

\theoremstyle{definition}

\newtheorem{remark}[theorem]{Remark}
\newtheorem{example}[theorem]{Example}

\newtheorem*{remark*}{Remark}

\numberwithin{equation}{section}

\def\eps{\varepsilon}

\def\Oeh{\mathcal{O}}
\def\Peh{\mathcal{P}}
\def\Reh{\mathcal{R}}

\def\N{\mathbb{N}}

\def\R{\mathbb{R}}

\def\Hh{\mathscr{H}}

\def\supp{\mathrm{supp\,}}

\def\Mbeta{G_\beta}

\newcommand{\loc}{\mathrm{loc}}
\DeclareMathOperator*{\diag}{diag}

\date{}

\begin{document}

\title{Regularity of all minimizers of a class of \\ spectral partition problems}

\author[H. Tavares]{Hugo Tavares}
\address{Hugo Tavares \newline \indent CAMGSD and Departamento de Matem\'atica, Instituto Superior T\'ecnico \newline \indent Pavilh\~ao de Matem\'atica, Av. Rovisco Pais \newline \indent
1049-001 Lisboa, Portugal}
\email{hugo.n.tavares@tecnico.ulisboa.pt}

\author[A. Zilio]{Alessandro Zilio}
\address{Alessandro Zilio \newline \indent (1) Universit\'e de Paris, Laboratoire Jacques-Louis Lions (LJLL), F-75013 Paris, France. \newline \indent (2) Sorbonne Universit\'e, CNRS, LJLL, F-75005 Paris, France}
\email{azilio@math.univ-paris-diderot.fr}

\date{\today}

\maketitle

\begin{center}
    \textit{To Sandro Salsa, with admiration and gratitude.}
\end{center}

\begin{abstract}
We study a rather broad class of optimal partition problems with respect to monotone and coercive functional costs that involve the Dirichlet eigenvalues of the partitions. We show a sharp regularity result for the entire set of minimizers for a natural relaxed version of the original problem, together with the regularity of eigenfunctions and a universal free boundary condition. Among others, our result covers the cases of the following functional costs
\[
	(\omega_1, \dots, \omega_m) \mapsto \sum_{i=1}^{m} \left( \sum_{j=1}^{k_i} \lambda_{j}(\omega_i)^{p_i}\right)^{1/p_i}, \quad  \prod_{i=1}^{m}  \left( \prod_{j=1}^{k_i} \lambda_{j}(\omega_i)\right), \quad  \prod_{i=1}^{m}  \left( \sum_{j=1}^{k_i} \lambda_{j}(\omega_i)\right)
\]
where $(\omega_1, \dots, \omega_m)$ are the sets of the partition and $\lambda_{j}(\omega_i)$ is the $j$-th Laplace eigenvalue of the set $\omega_i$ with zero Dirichlet boundary conditions.
\end{abstract}

\section{Introduction}

Let $\Omega \subset \R^N$ be a smooth bounded domain, $m\geq 2$ an integer and $k_1,\ldots, k_m\in \N$. Consider  the following optimal partition problem: among all $m$-tuples of open disjoint subsets $\omega_1, \dots, \omega_m$ of $\Omega$, belonging to an admissible class, find those that minimize the functional
\[
	(\omega_1, \dots, \omega_m) \mapsto F\left( \varphi_1\left(\lambda_1(\omega_1), \dots, \lambda_{k_1}(\omega_1)\right), \dots, \varphi_m\left(\lambda_1(\omega_m), \dots, \lambda_{k_m}(\omega_m)\right) \right) 
\]
where $\lambda_i(\omega)$ is the $i$-th eigenvalue of $\omega$ with Dirichlet boundary conditions. Here $F$ and $\varphi_i$ are given functions which satisfy certain monotonicity and coercivity assumptions. The aim of this paper is to show that not only problems of this form have a regular solution, but also that any solution is regular. Examples of functionals that fall in the scope of our results are
\begin{equation}\label{eq:firstexamples}
	(\omega_1, \dots, \omega_m) \mapsto \sum_{i=1}^{m} \left( \sum_{j=1}^{k_i} \lambda_{j}(\omega_i)^{p_i}\right)^{1/p_i}, \quad  \prod_{i=1}^{m}  \left( \prod_{j=1}^{k_i} \lambda_{j}(\omega_i)\right), \quad  \prod_{i=1}^{m}  \left( \sum_{j=1}^{k_i} \lambda_{j}(\omega_i)\right)
\end{equation}
and combinations of these functionals.

Optimal partition problems are a particular case of a shape optimization problem that appears quite naturally in engineering, where a cost functional defined on a structure made of several materials is being optimized (each material corresponds to a set of the partition).

 The problem of existence and regularity of optimal shapes for spectral costs (meaning cost functionals that depend on the spectrum of an operator set in a specific member of the partition) has been addressed by many authors.  They are connected with the study of nodal sets of eigenfunctions of Schr\"odinger operators \cite{BBHU,BKS, HHT,BHV}, monotonicity formulas \cite{ACF, CaffarelliLin2, ContiTerraciniVerziniOPP, HHOT2,TavaresTerracini1} and nonlinear systems of partial differential equations with strong competition terms \cite{CaffarelliLin2, clll,ctv2002, ContiTerraciniVerziniOptimalNonlinear, ContiTerraciniVerziniOPP, HHT, RamosTavaresTerracini, SoTe, TavaresTerracini2}. Moreover, these problems provide examples of monotone functionals which are lower-semicontinuous with respect to the weak $\gamma$-convergence, where existence results of a relaxed formulation (partitions of quasi-open sets) can be achieved by direct methods \cite{BucurButtazzoHenrot,BucurButtazzoBook}.  Alternative methods typically involve penalization arguments (see for instance \cite{BourdinBucurOudet,HHT, HHOT2, RamosTavaresTerracini,TavaresTerracini2}).

The main goal of this paper is to characterize and prove regularity of all possible partitions and their eigenvalues to problem \eqref{eqn prob}.

\subsection*{Open partitions} We contextualize our results by introducing a first natural formulation of the problem. For a given $m \geq 2$, consider the set of open partitions of $\Omega$ in $m$ disjoint subsets, denoted by
\[
    \Peh_m(\Omega)=\left\{(\omega_1,\ldots,\omega_m) : \ \omega_i \subset \Omega \text{ open }\forall i,\ \omega_i\cap \omega_j=\emptyset\ \forall i\neq j\right\}.    
\]
Observe that, according to this definition, a partition is not necessarily exhaustive, meaning that possibly $\cup_i \omega_i \subsetneqq \Omega$. To any element $\omega$ of a partition we associate the corresponding eigenvalues of the Laplacian with zero Dirichlet boundary condition $\lambda_1(\omega) \leq \lambda_2(\omega) \leq \dots$, counting multiplicity. It is well-known that these eigenvalues are the critical levels of the Rayleigh quotient
\[
	u \in H^1_0(\omega) \mapsto \left. \int_{\omega} |\nabla u|^2 \right/ \int_{\omega} u^2
\]
where $H^1_0(\omega)$ is the closure of the subset of $H^1(\Omega)$ whose support is contained in $\omega$. A characterization of eigenvalues, which takes naturally into account their multiplicity and is also better suited for our purpose, is given by the Courant-Fisher-Weyl formula, which states that for any $j \geq 1$ 
\[
	\lambda_j(\omega) = \mathop{\inf_{M\subset H^1_0(\omega)}}_{\dim M= j} \sup_{u\in M \setminus \{0\} }   \left( \left. \int_{\omega} |\nabla u|^2 \right/ \int_{\omega} u^2 \right)
\]
where $M$ is any linear subset of $H^1_0(\omega)$ of dimension $j$.

\subsection*{Cost functional} We introduce a general class of cost functional for the optimal partition problem. Let $F \in C^1(\R^m; \R)$ and, for any $i = 1, \dots, m$, $\varphi_i \in C^1((\R^+)^{k_i}; \R)$, functions that verify the following assumptions.
\begin{itemize}
\item[(H1)] Monotonicity: for every $i=1,\ldots,m$,
\[
    \begin{split}
    	\frac{\partial F}{\partial x_i}(x_1, \dots, x_m) > 0 \quad &\forall (x_1, \dots, x_m) \in (\R^+)^m,\\
    	 \quad \frac{\partial \varphi_i}{\partial x_j}(s_1, \dots, s_{k_i}) > 0 \quad &\forall (s_1, \dots, s_{k_i}) \in  (\R^+)^{k_i}, \, j \in 1, \dots, k_i;
    \end{split}
\]
\item[(H2)] Coercivity: for every $i=1,\ldots,m$, 
\[
    \begin{split}
        \lim_{t \to +\infty} F(x_1, \dots, x_{i-1}, t, x_{i+1}, \dots, x_m) = +\infty \quad & \forall (x_1, \dots, x_m) \in (\R^+)^m  \\
       \lim_{t \to +\infty} \varphi_{i}(s_1, \dots, s_{j-1}, t, x_{j+1}, \dots, s_{k_i}) = +\infty  \quad & \forall (s_1, \dots, s_{k_i}) \in (\R^+)^{k_i}, \, j \in 1, \dots, k_i;
    \end{split}
\]
\item[(H3)] Symmetry, for every $i=1,\ldots,m$ 
\[
    \varphi_i(\sigma(s_1, \dots, s_{k_i}))=\varphi_i(s_1, \dots, s_{k_i}) \qquad \text{for every permutation $\sigma\in S_{k_i}$.}
\]
\end{itemize}
We consider the following problem: among all partition $(\omega_1, \dots, \omega_m) \in \Peh_m(\Omega)$, find
\begin{equation}\label{eq:OPP_open}
    \inf_{(\omega_1,\ldots,\omega_m)\in \Peh_m(\Omega)} F\left(\varphi_1(\lambda_1(\omega_1),\ldots, \lambda_{k_1}(\omega_1)),\ldots, \varphi_m(\lambda_1(\omega_m),\ldots, \lambda_{k_m}(\omega_m))\right).
\end{equation}
The goal here is to show that a solution, an optimal partition, exists and also to establish some of its qualitative properties, such as the regularity of the associated eigenfunctions, topological properties of the partitions and the structure of their boundary.

Although this first formulation has a very natural appeal, it comes with an apparent incompatibility between the structure of the set of solutions $\Peh_m(\Omega)$ and the minimization problem. Indeed it does not seem easy to endow the set of the open partitions $\Peh_m(\Omega)$ with a topology that allows any compactness results on sequences of minimizers of the cost functional. There are many ways to circumvent this issue (see for instance \cite{BourdinBucurOudet, BucurButtazzoHenrot, HHT}), usually by considering a relaxed version of the original problem.

\subsection*{Measurable partitions} We adopt here the framework of \cite{HHT}, see also \cite{dePhilippisetal}, in that we reformulate our problem in the context of measurable sets. For this reason we extend our notion of partition and consider the set of measurable partitions of $\Omega$ in $m$ almost-disjoint subsets, denoted by
\[
    \widetilde \Peh_m(\Omega)=\left\{(\omega_1,\ldots,\omega_m):\ \omega_i \subset \Omega \text{ measurable }\forall i,\ |\omega_i\cap \omega_j|=0\ \forall i\neq j\right\},
\]
where $|\cdot|$ is the Lebesgue measure. Correspondingly, for any $\omega \subset \R^N$ measurable (with non-empty interior) we define the Sobolev-like set
\[
    \widetilde{H}^1_0(\omega) := \left\{u \in H^1(\Omega) : u = 0 \text{ a.e. in } \Omega \setminus \omega  \right\}
\]
and we introduce the generalized eigenvalues of $\omega$ as
\[
    \widetilde{\lambda}_j(\omega) := \mathop{\inf_{M\subset \widetilde H^1_0(\omega)}}_{\dim M= j} \sup_{u\in M \setminus \{0\} } \left( \left. \int_\omega |\nabla u|^2 \right/ \int_\omega u^2 \right).
\]
They form a nondecreasing sequence which is associated to an $L^2$--orthonormal sequence of eigenfunctions $\{\phi_j\}_{j\in \N}$, which satisfy $-\Delta \phi_j= \widetilde \lambda_j(\omega)  \phi_j $ in the weak sense
\[
\int_\Omega \nabla \phi \cdot \nabla \eta =  \widetilde \lambda_j(\omega)  \int_\Omega\widetilde \phi_j \eta \qquad \forall \eta\in \widetilde{H}^1_0(\omega)
\]
and belong to $L^\infty(\Omega)$ (see \cite[Section 2]{dePhilippisetal}).
\begin{remark*}
The notions of classical eigenvalue $\lambda_k$ and generalized eigenvalue $\widetilde{\lambda}_k$ differ in general, even for Lipschitz sets. Indeed, there are open sets $\Omega \subset \R^N$, such that $\lambda_k(\Omega) \neq \widetilde{\lambda}_k(\Omega)$ for some $k$ (in general we have $\widetilde{\lambda}_k(\Omega) \leq \lambda_k(\Omega)$). Taking for instance $\Omega = B_1(0) \setminus \{x_1 = 0\}$, then one easily verifies that $\lambda_1(\Omega) = \lambda_2(\Omega) = \lambda_2(B_1(0))$, while $\widetilde{\lambda}_k(\Omega) = \lambda_k(B_1(0))$ for any $k \in \N$. On the other hand, if $\Omega$ has smooth boundary (for instance, $\Omega$ enjoys an exterior cone condition), then the two notions coincide. See \cite{dePhilippisetal} for a more in depth discussion on this subject.
\end{remark*}

We can finally introduce a suitable relaxed formulation of the minimization problem: among all partition $(\omega_1, \dots, \omega_m) \in \widetilde{\Peh}_m(\Omega)$, find 
\begin{equation}\label{eqn prob}
    \inf_{(\omega_1,\ldots,\omega_m)\in \widetilde{\Peh}_m(\Omega)} F\left(\varphi_1 \left(\widetilde{\lambda}_1(\omega_1),\ldots, \widetilde{\lambda}_{k_1}(\omega_1)\right),\ldots, \varphi_m\left(\widetilde{\lambda}_1(\omega_m),\ldots, \widetilde{\lambda}_{k_m}(\omega_m)\right)\right).
\end{equation}

We state a general existence theorem for the solutions of this problem
\begin{theorem*}[\cite{RamosTavaresTerracini}]
The optimal partition problem \eqref{eqn prob} coincides with \eqref{eq:OPP_open} and admits an open regular solution $(\omega_{1},\ldots,\omega_{m})\in \Peh_m(\Omega)$ This partition is that $\widetilde \lambda_j(\omega_i)=\lambda_j(\omega_i)$ for every $i=1,\ldots, m$, $j=1,\ldots, k_i$. Moreover, there exist

Moreover, for all $i = 1, \dots, m$ there exist $k_i$ linearly independent eigenfunctions $u_{i,1},\ldots, u_{i,k_i} \in \widetilde{H}^1_0(\omega_i)$ associated to $\widetilde{\lambda}_1(\omega_i), \ldots, \widetilde{\lambda}_{k_i}(\omega_i)$ which are Lispschitz continuous, and $O_i$ coincides with the interior of the support of 
\[
\sum_{j=1}^{k_i} |u_{i,j}|.
\]
Finally, for each $i=1,\ldots, m$ and $j = 1, \dots, k_i$ there exists $a_{i,j}>0$ such that given $x_0$ in the regular part between the interface between two adjacent sets $O_p, O_q$ of the partition, the free boundary condition is 
\[
\mathop{\lim_{x\to x_0}}_{x\in  O_p} \sum_{j=1}^{k_p} a_{p,j} |\nabla u_{p,j}(x)|^2=\mathop{\lim_{x\to x_0}}_{x\in  O_q} \sum_{j=1}^{k_q} a_{q,j} |\nabla u_{q,j}(x)|^2 \neq 0.
\]
\end{theorem*}
For the notion of regular partition, we refer to the next statement. This statement is a combination of \cite[Theorem 1.2]{RamosTavaresTerracini} and the paragraphs after that, see in particular the relaxed formulation (2.4) therein.  It should be noted that the case of functionals depending only with first eigenvalues was treated in \cite{Alper, CaffarelliLin2, CaffLinStratification, ContiTerraciniVerziniOPP}, while \cite{TavaresTerracini2} deals with second eigenvalues.

\subsection*{Main results} In this paper we strengthen the previous result, by showing that every solution of \eqref{eqn prob} is equivalent to a regular partition, together with universal results regarding the regularity of eigenfunctions and a free boundary condition. In what follows, $\triangle$ denotes the symmetric difference between two sets.

\begin{theorem}\label{thm:main_result}
Let $ \omega:=( \omega_1,\ldots,  \omega_m)\in \widetilde \Peh_m(\Omega)$ be any minimizer of \eqref{eqn prob}.  Then there exists a unique open partition $O = (O_1, \ldots, O_m) \in \Peh_m(\Omega)$ such that the following holds.

\noindent\emph{Equivalence:}
\begin{itemize}
	\item subsets coincide up to negligible sets,  $|\omega_i \triangle O_i| = 0$ for all $i = 1, \dots, m$;
	\item they share the same eigenvalues, 
	\[
		\widetilde{\lambda}_j(\omega_i) = \lambda_j(O_i) \qquad \text{for all $i = 1, \dots, m$ and $j = 1, \dots, k_i$;}
	\]
	\item they share the same eigenfunctions, for all $i = 1, \dots, m$ there exist $k_i$ linearly independent eigenfunctions $\phi_{i,1},\ldots, \phi_{i,k_i} \in \widetilde{H}^1_0(\omega_i)$ associated to $\widetilde{\lambda}_1(\omega_i), \ldots, \widetilde{\lambda}_{k_i}(\omega_i)$ and $k_i$ linearly independent eigenfunctions $u_{i,1},\ldots,u_{i,k_i} \in H^1_0(O_i)$ associated to $\lambda_1(O_i), \ldots, \lambda_{k_i}(O_i)$ such that, for any $j \in 1, \dots, k_i$, we have
	\[
		\phi_{i,j} = u_{i,j} \qquad \text{quasi-everywhere in $\Omega$}.
	\]
\end{itemize}

\noindent\emph{Regularity of the sets:}
the partition $O$ is regular, in the sense that the free-boundary $\Gamma = \Omega \setminus \bigcup_{i=1}^{m} O_i$ is a rectifiable set and there exist disjoint sets $\Reh, \Sigma \subset \Gamma$ such that 
\begin{itemize}
	\item $\Gamma = \Reh \cup \Sigma$ has Hausdorff dimension at most $N-1$: $\Hh_\text{dim}(\Gamma)\leq N-1$;
	\item $\Reh$ is relatively open and $\Sigma$ is relatively close in $\Gamma$;
   	\item $\Reh$ is a collection of  hypersurfaces of class $C^{1,\alpha}$ (for some $0<\alpha<1$). Moreover, each hypersurface separates locally exactly two different elements of the partition: for every $x_0\in \Reh$, there exists $\rho>0$ and exactly two indices $i\neq j$ such that $x_0\in \partial O_i \cap \partial O_j$, $B_\rho(x_0)\setminus \Gamma= B_\eps(x_0) \cap ( O_i\cup O_j )$.

  	\item $\Sigma$ is small in the sense that $\Hh_\text{dim} (\Sigma)\leq N-2$;
	\item if $N=2$, the set $\Sigma$ is a locally finite set and $\Reh$ consists of a locally finite collection of curves meeting at singular points.
\end{itemize}	 

\noindent\emph{Spectral gap:}
\begin{itemize}
\item for each $i=1,\ldots, m$ it holds
\[
    \widetilde{\lambda}_{k_i}(\omega_i) < \widetilde{\lambda}_{k_i+1}(\omega_i).
\]
In particular, if $\widetilde E_{i,j}(\omega_i) \subset \widetilde H^1_0(\omega_i)$ denotes the eigenspace associated to $\widetilde \lambda_j(\omega_i)$, then the dimension of the linear space 
$E_{k_i}:=\mathrm{span}\left( \cup_{j=1}^{k_i} \widetilde E_{i,j}\right)$ is equal to $k_i$.
\end{itemize}

\noindent\emph{Regularity of the eigenfunctions:}
\begin{itemize}
\item for $i=1,\ldots, m$, we have
\[
E_{k_i} \subset Lip(\overline \Omega),
\] 
in the sense that each eigenfunction has a continuous representative.
\end{itemize}
Now, for $i=1,\ldots, m$, let $\phi_{i,1},\ldots, \phi_{i,k_{i}}$ be an $L^2$-orthonormal base of $E_{k_i}$, associated respectively to the eigenvalues $\widetilde \lambda_1(\omega_i)\leq \ldots\leq \widetilde \lambda_{k_i}(\omega_i)$. Then 
\begin{itemize}
\item for each $i=1,\ldots, m$, $O_i$ is the interior of the support of 
\[
	\sum_{j=1}^{k_i} |\phi_{i,j}|;
\]
\item there exists $a_{i,j}>0$ such that given $x_0 \in \Reh$ and $O_p, O_q$ the two adjacent sets of the partition at $x_0$, then
\begin{equation}\label{eq:extremality_main}
\mathop{\lim_{x\to x_0}}_{x\in  O_p} \sum_{n=1}^{k_p} a_{p,j} |\nabla \phi_{p,j}(x)|^2=\mathop{\lim_{x\to x_0}}_{x\in  O_q} \sum_{n=1}^{k_q} a_{q,j} |\nabla \phi_{q,j}(x)|^2 \neq 0.
\end{equation}
The coefficients depend only on the eigenvalues of the optimal partition, through the formula
\begin{multline*}
a_{i,j}=\partial_i F\left(\varphi_1 \left(\widetilde{\lambda}_1(\omega_1),\dots, \widetilde{\lambda}_{k_1}(\omega_1)\right),\dots, \varphi_m\left(\widetilde{\lambda}_1(\omega_m),\dots, \widetilde{\lambda}_{k_m}(\omega_m)\right)\right)\partial_j \varphi_i(\widetilde \lambda_1(\omega_1),\dots, \widetilde \lambda_k(\omega_1)),
\end{multline*}
and $a_{i,m}=a_{i,n}$ if $\widetilde \lambda_m(\omega_i)=\widetilde \lambda_n(\omega_i)$.
\end{itemize}
\end{theorem}

The proof of Theorem \ref{thm:main_result} is based on a penalization argument. We exploit a regularized version of the relaxed formulation \eqref{eqn prob}, involving eigenfunctions rather than eigenvalues, that is better suited to prove the aforementioned properties of optimal sets. Following \cite{RamosTavaresTerracini}, we consider a singular perturbation and  approximate these eigenfunctions by minimal solutions of a nonlinear elliptic system with competition terms. This allows to prove the regularity results concerning eigenfunctions and interfaces. By adding an extra term in the energy functional we are able to select any specific minimizer of which we wish to show regularity.

It should be noted that the previous result in the case of functionals depending on first eigenvalues was proved in \cite{HHT}. The case of higher eigenvalues presents many extra difficulties which are related to the unknown multiplicity of the eigenvalues of an optimal partition and to the fact that some eigenfunctions change sign.

\subsection*{Examples} Before presenting the proof of our result, we illustrate a couple of concrete applications for specific choices of cost functionals. As a model case, we consider the first function in \eqref{eq:firstexamples}, that is the case of 
\[
	F(x_1, \dots, x_m) = \sum_{i=1}^m x_i \quad \text{and} \quad \varphi_i(s_1, \dots, s_{k_i}) = \left( \sum_{i=1}^{k_i} s_j^{p_i}\right)^{\frac{1}{p_i}}
\]
with $p_i > 0$. Then our theory applies to all minimizer of 
\begin{equation*}
    \inf_{(\omega_1,\ldots,\omega_m)\in \widetilde \Peh_m(\Omega)} \sum_{i=1}^{m} \left( \sum_{j=1}^{k_i} \widetilde \lambda_{j}(\omega_i)^{p_i}\right)^{1/p_i},
\end{equation*}
which are the shown to be regular in the sense of Theorem \ref{thm:main_result}. Moreover, the coefficient in \eqref{eq:extremality_main} are given in this case by
\[
	a_{i,j} := \frac{\widetilde \lambda_n(\omega_i)^{p_i-1}}{\left(\sum_{j=1}^{k_i}\widetilde \lambda_{j}(\omega_1)^{p_i}\right)^\frac{p_i-1}{p_i}}.
\]
The same results also holds for the (suitably renormalized) limit case $p_i \to 0$, where we find
\[
    \varphi'_i(s_1, \dots, s_{k_i}) = \prod_{i=1}^{k_i} s_j, \quad \text{and} \quad a_{i,j}' := \prod_{j=1, j \neq i}^{k_i}\widetilde \lambda_{j}(\omega_1).
\]

\subsection*{A remark about quasi-open sets}
In the theory of optimal partitions with respect to spectral costs we can find another class of partitions, given by quasi-open sets, which is in a sense intermediate between the class of open partitions and the class of measurable partitions. It is defined by
\[
\widehat \Peh_m(\Omega)=\left\{(\omega_1,\ldots,\omega_m):\ \omega_i \subset \Omega \text{ quasi-open }\forall i,\ \textrm{cap}(\omega_i\cap \omega_j)=0\ \forall i\neq j\right\},
\]
with associated problem
\[
    \inf_{(\omega_1,\ldots,\omega_m)\in \widehat \Peh_m(\Omega)} F\left(\varphi_1(\lambda_1(\omega_1),\ldots, \lambda_{k_1}(\omega_1)),\ldots, \varphi_m(\lambda_1(\omega_m),\ldots, \lambda_{k_m}(\omega_m))\right).
\]
We recall briefly the notions of capacity and of quasi-open sets, taken from \cite[Chapter~4]{BucurButtazzoBook}. The (2-)capacity of a set is
\[
\textrm{cap}(A)=\inf \left\{ \int_\Omega (|\nabla u|^2+u^2):\ u\in H^1(\R^N),\ u\equiv 1 \text{ in a neighborhood of } A \right\}.
 \]
 A set $A$ is said to be quasi-open if for each $\eps>0$ there is an open set $A_\eps$ satisfying $\textrm{cap}(A \triangle A_\eps)<\eps$, where $\triangle$ denotes the symmetric difference between sets. There is a close relation between quasi-open sets and Sobolev functions. In fact, each $u\in H^1(\R^N)$ admits a quasi-continuous representative, this meaning that for each $\eps>0$ there is a continuous function $u_\eps$ with $\textrm{cap}(\{u\neq u_\eps\})<\eps$. Now $A$ is a quasi-open set if and only if $A=\{u>0\}$ for a quasi-continuous function $u$. It follows from the definition that, in the setting of this paper, any open minimal partition is a quasi-open minimal partition, and any quasi-open minimal partition is a measurable minimal partition. Then, thanks to Theorem \ref{thm:main_result}, we find that the three formulations are actually equivalent (up to negligible sets).

\subsection*{Numerical simulations and open problems}
We conclude this introduction providing some numerical simulations. They were obtained implementing the construction in Section 3 (see also \cite{RamosTavaresTerracini}), via a point fix iteration and a finite element discretization. All the simulations were implemented in \texttt{FreeFem++} \cite{freefem}, a free software available at \url{https://freefem.org/}.

In Figure \ref{fig sum}, a numerical approximation of the optimal partition of the unit ball associated to the cost functionals
\begin{equation}\label{func sum}
	(\omega_1, \omega_2) \mapsto \lambda_1(\omega_1) + \lambda_2(\omega_1) + \lambda_1(\omega_2) + \lambda_2(\omega_2j)
\end{equation}
and 
\begin{equation*}
	(\omega_1, \omega_2) \mapsto \lambda_1(\omega_1) \lambda_2(\omega_1) \lambda_1(\omega_2) \lambda_2(\omega_2)
\end{equation*}
The two functionals share, numerically, the same optimal partition. The first functional \eqref{func sum} is linear, making the algorithm quite efficient in this case.

\begin{figure}[h!]
	\includegraphics[width=.4\textwidth]{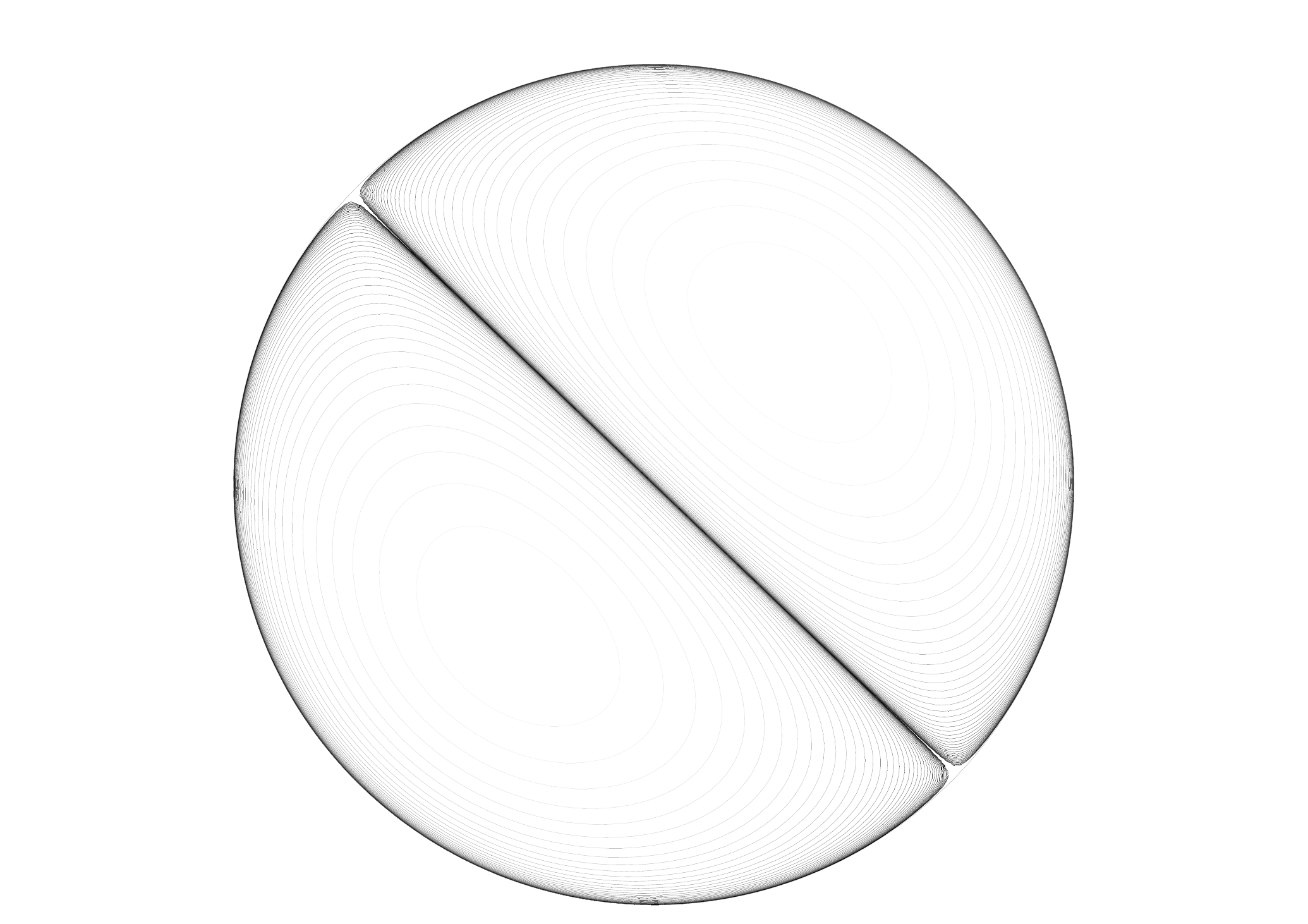} \hspace{.2cm}
	\includegraphics[width=.4\textwidth]{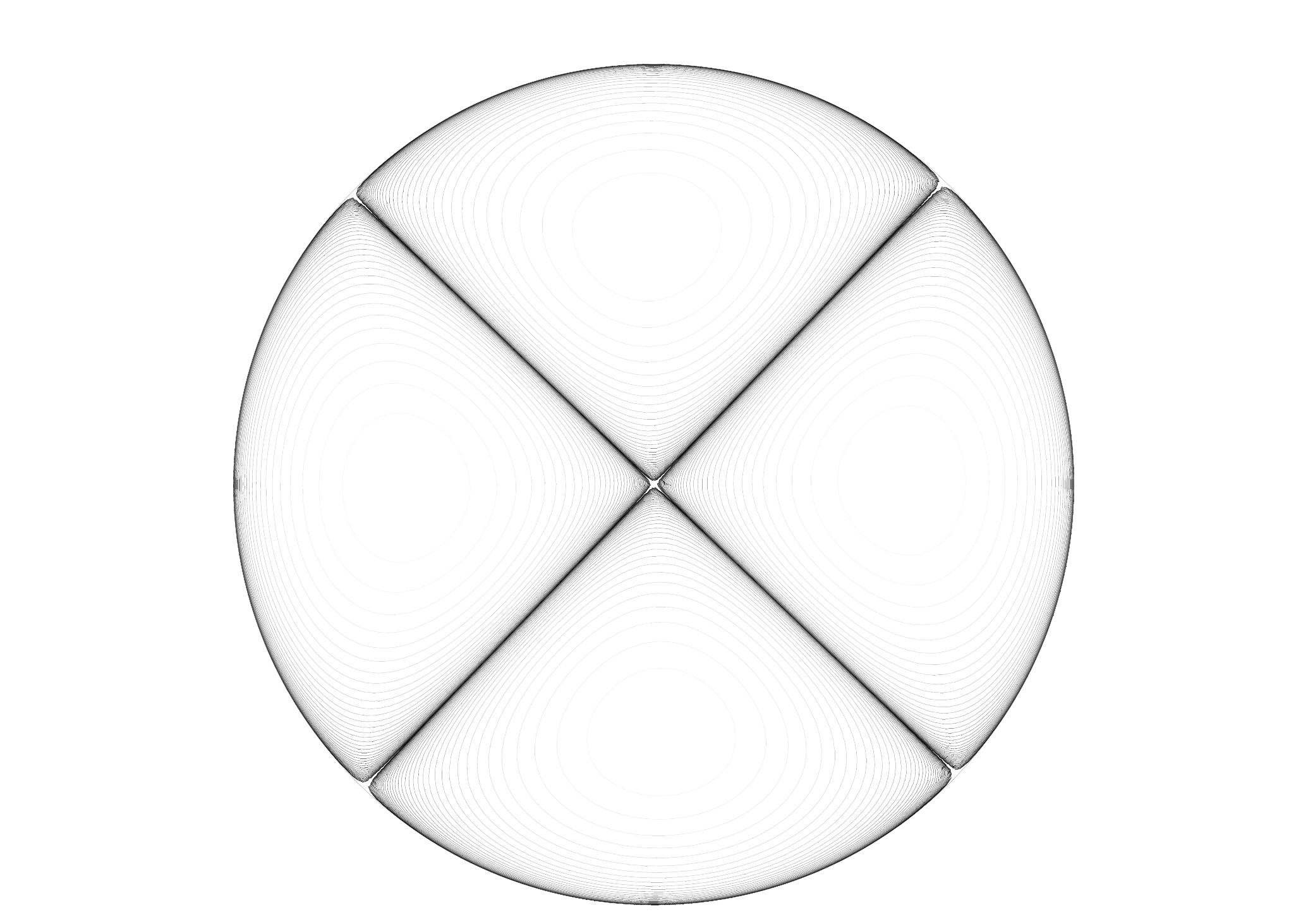} 
	\caption{Optimal partition and eigenfunctions for \eqref{func sum}. }
	\label{fig sum}
\end{figure}
On the left of Figure \ref{fig sum} is a representation of the eigenfunctions associated to the first eigenvalues of the partition: they highlight the two sets of the partition, which are symmetric semicircles. On the right the second eigenfunctions of the two sets. Observe the additional nodal lines (in connected sets the second eigenfunctions is sign-changing). In this case the strong symmetry of the two functionals seems to translate in the symmetry of their solutions.

In Figure \ref{fig detplus2}, a numerical approximation for
\begin{equation}\label{eqn detplus2}
	(\omega_1, \omega_2) \mapsto \lambda_1(\omega_1) \lambda_2(\omega_1) + \lambda_1(\omega_2)^2 + \lambda_2(\omega_2)^2.
\end{equation}
In this case the functional is no more symmetric and the solution too looses symmetry. Nevertheless, observe that the cost functional is scale-invariant.
\begin{figure}[h!]
	\includegraphics[width=.4\textwidth]{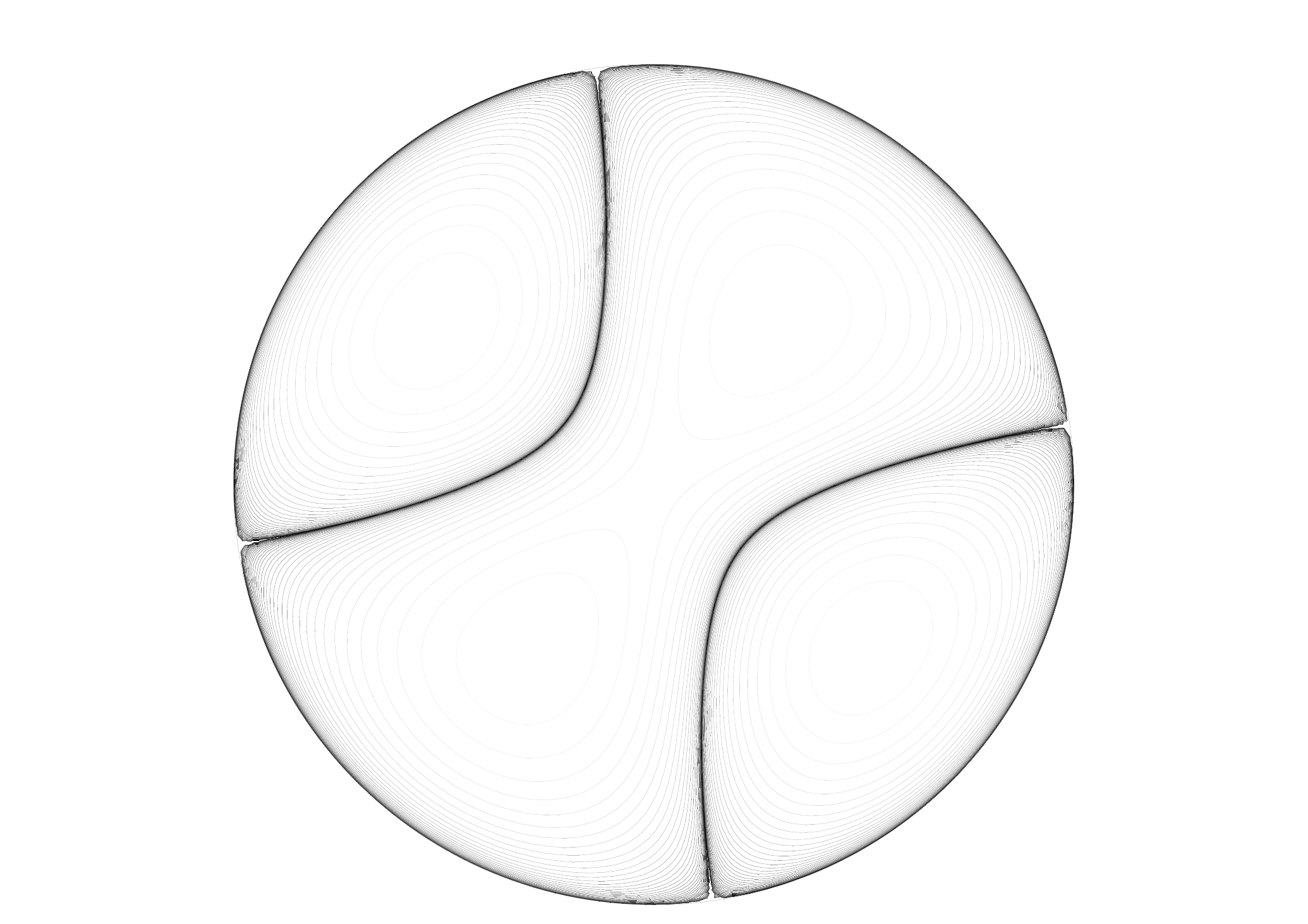} \hspace{.2cm}
	\includegraphics[width=.4\textwidth]{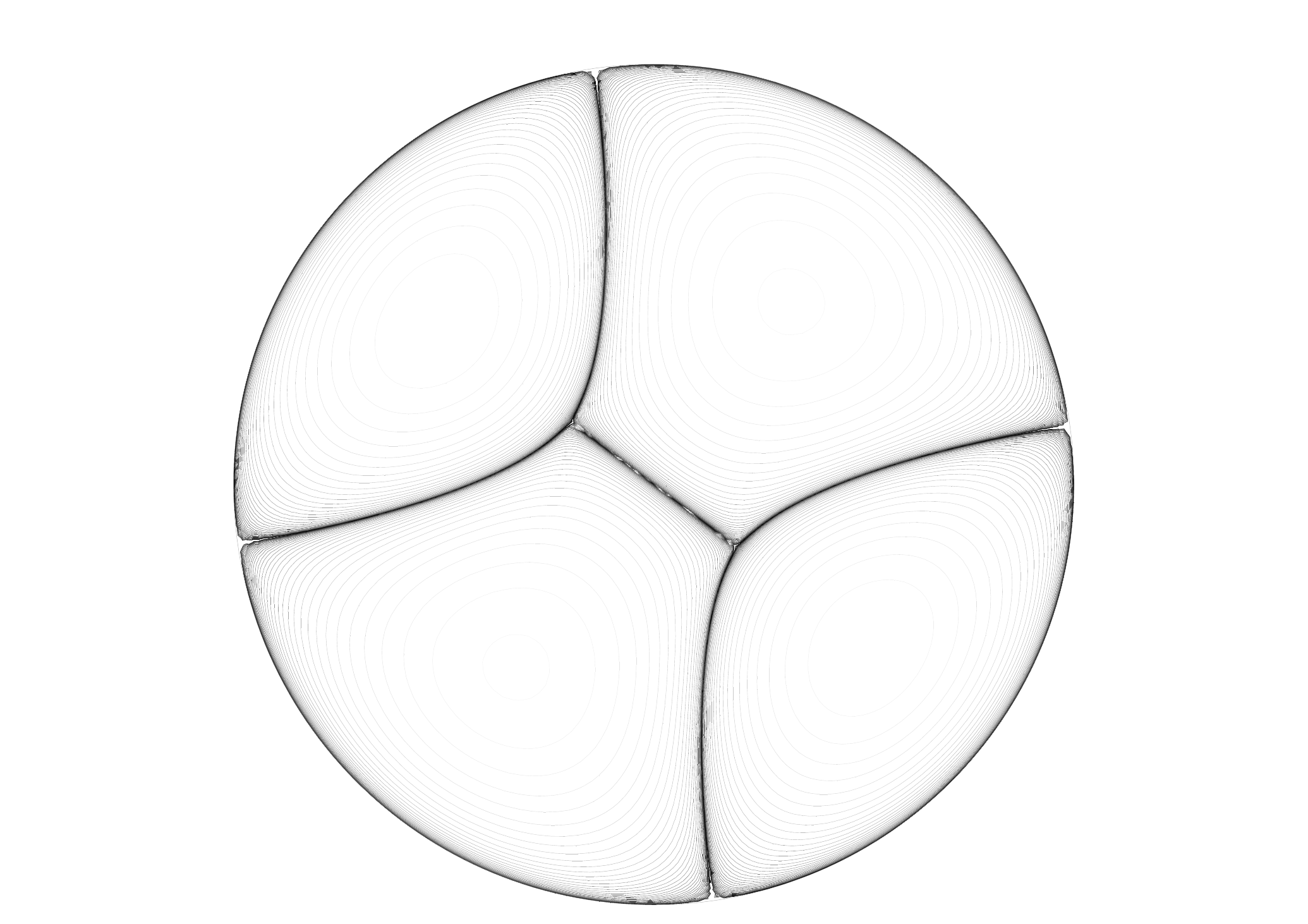} 
	\caption{Optimal partition and eigenfunctions for \eqref{eqn detplus2}. }
	\label{fig detplus2}
\end{figure}

On the left of Figure \ref{fig detplus2} is a representation of the eigenfunctions associated to the first eigenvalues of the partition and the two sets $\omega_1$ (in the center) and $\omega_2$ (the two lobes). On the right the second eigenfunctions of the two sets. Observe that the second domain is not connected and, numerically, it holds $\lambda_1(\omega_2) = \lambda_2(\omega_2)$. This implies that the first eigenvalue of the second subset of the partition has multiplicity two and one can choose the corresponding eigenfunction to have disjoint supports contained in only one of the two lobes at the time. This suggests that there are minimal partitions made of disconnected sets and were the eigenvalues have multiplicity higher than one (unlike the case of cost functions depending on first eigenvalues only).  Any choice of eigenfunctions will still verify \eqref{eq:extremality_main} with the same coefficients. Finally we point out that in this example the equi-partition of angles at singular points seems false (unlike in \cite{HHT}), although at the moment we lack any explicit counterexample of this fact.

In Figure \ref{fig max121}, a numerical approximation of the optimal partition of the unit ball  associated to
\begin{equation}\label{eqn max121}
	(\omega_1, \omega_2) \mapsto \left(\lambda_1(\omega_1)^{20} + \lambda_2(\omega_1)^{20} + \lambda_1(\omega_2)^{20}\right)^{1/20}.
\end{equation}
This functional gives a rather good approximation of the cost
\[
	(\omega_1, \omega_2) \mapsto \mathrm{max} \left(\lambda_2(\omega_1), \lambda_1(\omega_2)\right)
\]
which does not fall in the scope of our main result, as it is not strictly monotone with respect to $\lambda_1(\omega_1)$. It can be shown that the optimal partition corresponding to this last function is the two third sector of the circle ($\omega_1$) and a third sector of the circle ($\omega_2$). We obtain a rather similar result for \eqref{eqn max121}.
\begin{figure}[h!]
	\includegraphics[width=.4\textwidth]{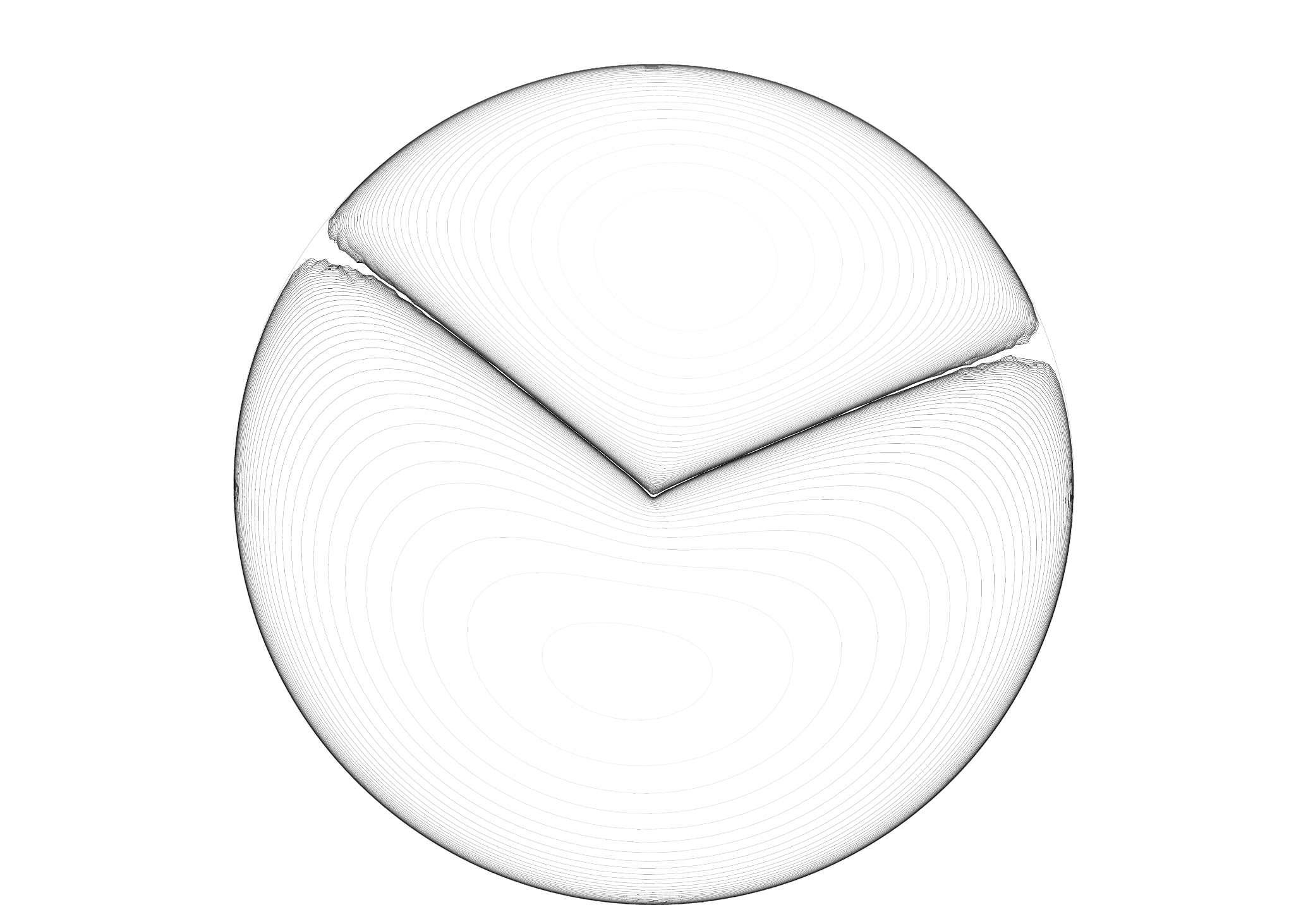} \hspace{.2cm}
	\includegraphics[width=.4\textwidth]{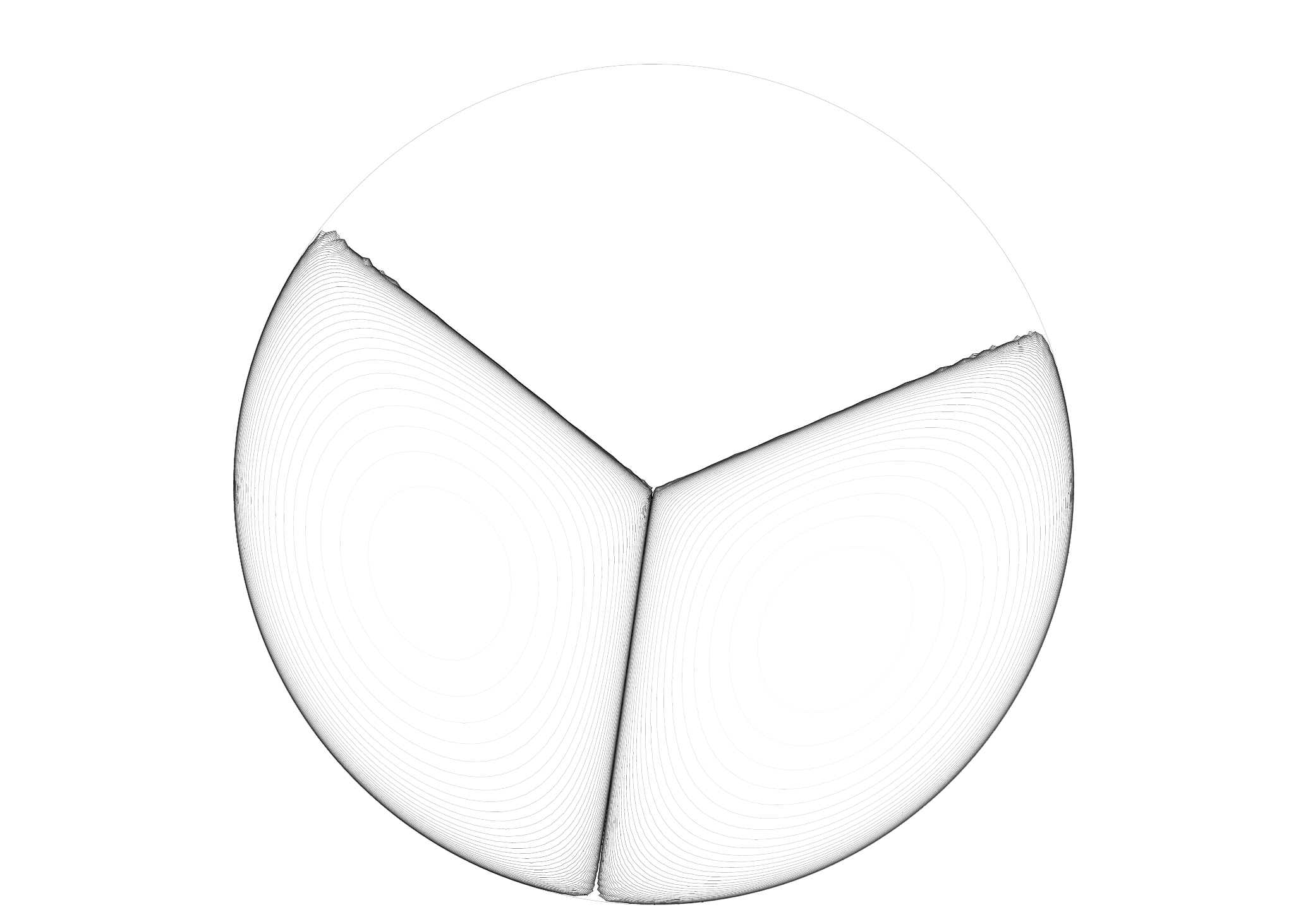} 
	\caption{Optimal partition and eigenfunctions for \eqref{eqn max121}. We observe some numerical artifact in the first picture: the presence of a region where the eigenfunctions are zero. This points out a weakness of our numerical scheme when some of the coefficients in \eqref{eq:extremality_main} are small compared to the others.}
	\label{fig max121}
\end{figure}
On the left the eigenfunctions associated to the first eigenvalues of the partition and on the right the second eigenfunction of the first subset. We point out a seemingly singular point at the center of the ball. According to Theorem \ref{thm:main_result} all the eigenfunctions in the energy functional are regular, and indeed the first eigenfunction of $\omega_1$ is regular, but it appears that as the exponent in the functional becomes larger and larger (the $l^p$ norm approaches the $l^{\infty}$ norm), the first eigenfunctions loses its regularity. This phenomenon will be the object of an upcoming paper.

\section{The penalization argument: an approximate problem}

In order to simplify the presentation, we only detail the proof in the case $m=2$, $k_1=k_2=:k\in \N$ and $\varphi_1=\varphi_2=:\varphi$. The general case follows by the same argument with some simple modifications. In this particular situation, problem \eqref{eqn prob} becomes
\begin{equation}\label{eq:OPPm=2_quasiopen}
    \widetilde c = \inf_{(\omega_1,\omega_2)\in \widetilde \Peh_2(\Omega)} F\left(\varphi(\widetilde \lambda_1(\omega_1),\ldots, \widetilde \lambda_{k}(\omega_1)),\varphi(\widetilde \lambda_1(\omega_2),\ldots, \widetilde \lambda_{k}(\omega_2))\right)
\end{equation}
where, we recall,
\[
\widetilde \Peh_2(\Omega)=\left\{(\omega_1,\omega_2)\subseteq \Omega\times \Omega:\ \omega_1,\omega_2 \text{ measurable, } |\omega_1\cap \omega_2|=0\right\}.
\]
Following \cite{RamosTavaresTerracini}, this problem has at least one open and regular solution in the sense of Theorem \ref{thm:main_result}. Here we show that \emph{every} solution of this problem is equivalent to an open and regular partition, together with some regularity properties of the associated eigenfunctions and a free boundary condition (\ref{thm:main_result}).

Keeping this in mind, let $( \omega_1, \omega_2)\in \widetilde \Peh_2(\Omega)$ be a solution of \eqref{eq:OPPm=2_quasiopen}. We denote by $\{(\widetilde \lambda_i(\omega_1), \phi_i)\}_{i \in \N}$ and $\{(\widetilde \lambda_i(\omega_2),\psi_i)\}_{i \in \N}$ the sequences of nondecreasing generalized eigenvalues (enumerated with multiplicity) and corresponding orthonormal eigenfunctions of the Laplacian in $\widetilde H^1_0(\omega_1)$ and $\widetilde H^1_0(\omega_2)$, respectively. We point out that, even though the eigenfunctions associated to the generalized eigenvalues belong to some Sobolev-like spaces, they are still $H^1_0(\Omega)$ functions. Thus we have the identities
\[
    \int_\Omega \phi_i \phi_j = \delta_{ij}, \quad \text{and} \quad\int_\Omega \nabla \phi_i \cdot \nabla  \phi_j = \widetilde \lambda_{i}( \omega_1) \delta_{ij},
\]
and similarly for $\{\psi_i\}_{i \in \N}$. Here $\delta_{ij}$ denotes the Kronecker symbol, that is $\delta_{ij}=1$ if $i=j$ and $0$ otherwise. 

\begin{remark}
We point out that, a priori, the sets $\mathrm{span}\{\phi_1,\ldots, \phi_k\}$ and $\mathrm{span}\{\psi_1,\ldots, \psi_k\}$ may not contain all the eigenfunctions associated to $\widetilde \lambda_k( \omega_1)$ and $\widetilde \lambda_k( \omega_2)$. However, we shall see later on that this is never the case, thanks to the spectral gap property (cfr.\ Theorem \ref{thm:main_result}).
\end{remark}
 
We denote $\boldsymbol{\phi}=(\phi_1,\ldots, \phi_k)$ and $\boldsymbol{\psi}=(\psi_1,\ldots, \psi_k)$
and we introduce two linear subspaces of $L^2(\Omega)$ generated by $\boldsymbol{\phi}$ adn $\boldsymbol{\psi}$, together with their orthogonal projections:
    \begin{align*}
        &L(\boldsymbol{\phi})=\mathrm{span}\left\{\phi_1,\ldots, \phi_k \right\}, \qquad P^\perp : L^2(\Omega) \to L(\boldsymbol{\phi})^\perp, \\ 
        &L(\boldsymbol{\psi})=\mathrm{span}\left\{\psi_1,\ldots, \psi_k \right\},  \qquad Q^\perp : L^2(\Omega) \to L(\boldsymbol{\psi})^\perp.
    \end{align*}
Exploiting the orthogonality of $\boldsymbol{\phi}$ and $\boldsymbol{\phi}$ we find that  for every $w\in L^2(\Omega)$ the projections are
\[
	P^\perp w=w-\sum_{i=1}^k \langle w,\phi_i\rangle_{L^2(\Omega)} \qquad \text{and} \qquad Q^\perp w=w-\sum_{i=1}^k \langle w,\psi_i\rangle_{L^2(\Omega)}.
\]
where $\langle \cdot,\cdot \rangle_{L^2(\Omega)}$ denotes the usual scalar product in $L^2(\Omega)$.

Our aim is to define an energy functional and an associated minimization problem whose solutions are close to those of \eqref{eq:OPPm=2_quasiopen}. In order to achieve this, we need to introduce a regularized energy functional with two additional terms. For the first one, inspired by \cite{RamosTavaresTerracini}, we relax the disjointedness constraint of the supports of the eigenfunctions $\boldsymbol{\phi}$ and $\boldsymbol{\psi}$ by introducing a competition term between groups of eigenfunctions; this allows to prove the regularity of both the partition and of the eigenfunctions. For the second one, using the projection operators $P^\perp$ and $Q^\perp$, we introduce a penalization that enables us to select the specific minimizer to which the sequence of approximated minimizers converges. This allows to prove that the singular limits are, up to orthogonal transformation, the original eigenfunctions. We need a couple of technical tools before introducing the approximating functionals.

Given $\mathbf{u}, \mathbf{v}\in H^1_0(\Omega; \R^k)$, define the $k \times k$ symmetric and positive definite matrices 
\begin{equation*}
    \begin{split}
        M(\mathbf{u}):=&\left( \int_\Omega \nabla u_i\cdot \nabla u_j\, + (P^\perp u_i)(P^\perp u_j) \right)_{i,j} = \left( \langle \nabla u_i , \nabla u_j \rangle_{L^2(\Omega)} + \langle P^\perp u_i, P^\perp u_j \rangle_{L^2(\Omega)} \right)_{i,j},\\
        N(\mathbf{v}):=&\left( \int_\Omega \nabla v_i\cdot \nabla v_j\, + (Q^\perp v_i)(Q^\perp v_j) \right)_{i,j} = \left( \langle \nabla v_i, \nabla v_j \rangle_{L^2(\Omega)} + \langle Q^\perp v_i, Q^\perp v_j \rangle_{L^2(\Omega)} \right)_{i,j}.
    \end{split}
\end{equation*}
Observe that for any orthogonal matrix $O\in \Oeh_k(\R)$ we have
\[
    M( O \mathbf{u}) = O M(\mathbf{u}) O^T , \qquad N( O \mathbf{v}) = O N(\mathbf{v}) O^T.
\]
In particular $M( O \mathbf{u})$ and $M(\mathbf{u})$ have the same spectrum.

We extend the function $\varphi:(\R^+)^k\to \R$ to the set of symmetric and positive definite matrices in the following way: given such a matrix $M$, we let 
\[
\varphi(M)= \varphi(\gamma_1,\ldots, \gamma_k),
\]
where $\gamma_1,\dots, \gamma_k$ are the (positive) eigenvalues of $M$ (with an abuse of notation, we identify the function acting on the eigenvalues with the function acting on the matrices). Observe that such function is well defined by the symmetry assumption (H3). By definition, we have 
\[
\varphi(OMO^T)=\varphi(M) \qquad \text{ for every $M$ symmetric positive definite, $O\in \mathcal{O}_k(\R)$}
\]
Since the original function (acting on the eigenvalues) is smooth and symmetric, we find that $\varphi$ is also a $C^1$ function in the set of symmetric and positive definite matrices. We denote
\[
	\frac{\partial}{\partial E_{ij}} \varphi(M) = \lim_{h \to 0} \frac{\varphi(M + h (E_{ij}+E_{ji})/2 ) - \varphi(M)}{h}
\]
the (tangent) derivative, in the set of symmetric matrices,  of $\varphi$ at $M$ with respect to the component $(i,j)$. Here $E_{ij}$ is the matrix whose component $(i,j)$ is equal to $1$, while all other components are $0$.
\begin{example} In some notable cases the extended functions can be computed explicitly. For the map $(s_1,\ldots,s_k)\mapsto \left(\sum_{i=1}^k (s_i)^p \right)^{1/p}$, we have $\varphi(M):=\left(\sum_{i=1}^k(\gamma_i)^p\right)^{1/p} =\left(\textrm{trace}(M^p)\right )^{1/p} $,
which coincides with the $p$-Schatten norm of a symmetric and positive definite matrix $M$. For $(s_1,\ldots, s_k)\mapsto \prod_{i=1}^k s_i$, we have $\varphi(M):= \prod_{i=1}^k \gamma_i= \det(M)$. These examples are related to \ref{eq:firstexamples}.
\end{example}

\begin{lemma}[{{\cite[Lemma 3.6]{RamosTavaresTerracini}}}]\label{lemma:derivative_0_outsidediagonal}
For every diagonal matrix $D=\diag(\gamma_1,\ldots, \gamma_k)$, we have
\[
\frac{\partial}{\partial E_{ii}} \varphi(D)=\partial_i \varphi (\gamma_1,\ldots, \gamma_k) \quad \forall i,\qquad \frac{\partial}{\partial E_{ij}} \varphi(D)=0 \quad \forall i\neq j.
\]
\end{lemma}

We are now ready to introduce the family of approximating functionals. Fix any exponent $1/2< q <2^*/4 = N/[2(N-2)^+]$. For $\beta>0$ we define the $C^1$ energy functional $E_{\beta}:H^1_0(\Omega,\R^k)\times H^1_0(\Omega,\R^k)\to \R$ as
\begin{align*}
E_{\beta}(\mathbf{u},\mathbf{v})=F\left(\varphi(M(\mathbf{u})),\varphi(N(\mathbf{v}))\right)+\frac{\beta}{q}\int_\Omega \Big(\sum_{i=1}^k u_i^2\Big)^q \Big(\sum_{i=1}^k v_i^2\Big)^q
\end{align*}
and the least energy level
\begin{equation}\label{eq:c_p_beta}
c_{\beta}:=\inf \left\{ E_{\beta}(\mathbf{u},\mathbf{v}): \mathbf{u}, \mathbf{v}\in \Sigma(L^2) \right\},
\end{equation}
where
\[
	\Sigma(L^2):=\left\{ \mathbf{w} =(w_1,\ldots, w_k)\in H^1_0(\Omega; \R^k):\ \int_\Omega w_i w_j =\delta_{ij} \text{ for every }i,j \right\}.
\]
The functional and the set $\Sigma(L^2)$ are invariant under multiplication by orthogonal matrices
\[
	E_{\beta}(\mathbf{u}, \mathbf{v}) = E_{\beta}(O_1\mathbf{u}, O_2\mathbf{v}) \qquad \forall O_1, O_2 \in \mathcal{O}_k(\R),
\]
and
\[
	(\mathbf{u}, \mathbf{v}) \in \Sigma(L^2) \iff (O_1\mathbf{u}, O_2\mathbf{v}) \in \Sigma(L^2), \;\forall O_1, O_2 \in \mathcal{O}_k(\R).
\]
One should keep in mind that $E_{\beta}$ and $c_{\beta}$ also depend on the vectors of eigenfunctions $\boldsymbol{\phi}$, $\boldsymbol{\psi}$. However, in order to simplify the notation, we will not point out this dependence explicitly.

\begin{lemma}\label{lemma:easystatements}
For each $\beta>0$ we have
\[
F(\varphi(M(\mathbf{u}),\varphi(N(\mathbf{v})))\geq F(\varphi(\lambda_1(\Omega),\lambda_k(\Omega)),\ldots, \varphi(\lambda_1(\Omega),\ldots, \lambda_k(\Omega)))\qquad \forall \mathbf{u},\mathbf{v}\in \Sigma(L^2)
\]
and $c_\beta$ is finite.
\end{lemma}
\begin{proof}
For any $(\mathbf{u}, \mathbf{v}) \in \Sigma(L^2)$, take $O_1,O_2\in \mathcal{O}_k(\R)$ in such a way that $O_1M(\mathbf{u})O_1^T$, $O_2 N(\mathbf{v})O_2^T$ are diagonal and the elements on the diagonal are ordered nondecreasingly. Let $\mathbf{\widetilde u}=O_1 \mathbf{u}$, $ \mathbf{\widetilde v}=O_2 \mathbf{v}$. Exploiting the monotonicity of $F$ and $\varphi$, and the invariance of $\Sigma(L^2)$ and $\varphi$ under orthogonal transformations, we find that
\begin{align*}
F(\varphi(M(\mathbf{u}),\varphi(N(\mathbf{v})))&= F(\varphi(M(\mathbf{\widetilde u}),\varphi(N(\mathbf{\widetilde v})))\\
&= F\left(\varphi\left(\int_\Omega |\nabla \widetilde u_1|^2+(P^\perp \widetilde u_1)^2,\ldots, \int_\Omega |\nabla \widetilde u_k|^2+(P^\perp \widetilde u_k)^2\right),\right.\\
								&\phantom{=F(\varphi (\int} \left.  \varphi\left(\int_\Omega |\nabla \widetilde v_1|^2+(Q^\perp \widetilde v_1)^2, \ldots, \int_\Omega |\nabla \widetilde v_k|^2+(Q^\perp \widetilde v_k)^2\right)\right)\\
									&\geq F(\varphi(\lambda_1(\Omega),\ldots, \lambda_k(\Omega)),\ldots, \varphi(\lambda_1(\Omega),\ldots, \lambda_k(\Omega)))
\end{align*}
Then, recalling that $\beta>0$, we conclude 
\[
c_\beta\geq F(\varphi(\lambda_1(\Omega),\ldots, \lambda_k(\Omega)),\ldots, \varphi(\lambda_1(\Omega),\ldots, \lambda_k(\Omega))) >-\infty.\qedhere
\]
\end{proof}

We have established that for any $\beta > 0$, the functional $E_\beta$ is bounded from below in $\Sigma(L^2)$. We now show that the infimum is always attained, making the least energy level $c_\beta$ in \eqref{eq:c_p_beta} a critical level for $E_\beta$. For notation convenience, let
\begin{equation*}
	\Mbeta = \{(\mathbf{u}, \mathbf{v}) \in \Sigma(L^2) : E_{\beta}(\mathbf{u}, \mathbf{v}) = c_\beta\}.
\end{equation*}

\begin{proposition}\label{thm:minimizer_for_c_p}
For any $\beta>0$, we have the following: 
\begin{enumerate}
\item[(a)] the value $c_{\beta}$ is a critical level for the functional $E_{\beta}$ and $\Mbeta$ is  not empty. Moreover, for every $(\mathbf{u},\mathbf{v})=((u_{1},\ldots, u_{k}),(v_{1},\ldots, v_{k}))\in M_\beta$, we have
\[
E'_{\beta}(\mathbf{u}, \mathbf{v})=0.
\]
\item[(b)]  For any $O_1, O_2 \in \Oeh_k(\R)$ orthogonal matrices,
\[
	(\mathbf{u}, \mathbf{v}) \in \Mbeta \implies (O_1 \mathbf{u}, O_2 \mathbf{v}) \in \Mbeta.
\]
Therefore, if $(\mathbf{u}, \mathbf{v}) \in \Mbeta$ we can further assume that it verifies
\begin{align}
    \int_\Omega \nabla u_i\cdot \nabla u_j + (P^\perp u_i) (P^\perp u_j) &=\int_\Omega \nabla v_i\cdot \nabla v_j + (Q^\perp v_i) (Q^\perp v_j) =0& \forall i\neq j  \label{eq:orthogonality} \\
    \int_\Omega |\nabla u_i|^2 + (P^\perp u_i)^2  \leq \int_\Omega |\nabla u_j|^2 + (P^\perp u_j)^2 &, \; \int_\Omega |\nabla v_i|^2 + (Q^\perp v_i)^2 \leq \int_\Omega |\nabla v_j|^2 + (Q^\perp v_j)^2 & \forall i\leq j. \label{eq:monotonicity}
\end{align}
In particular, $M(\mathbf{u}),N(\mathbf{v})$ are orthogonal matrices, and
\begin{multline*}
E_{\beta}(\mathbf{u},\mathbf{v})=F\left( \varphi\left(\int_\Omega |\nabla u_1|^2 + (P^\perp u_1)^2,\ldots, \int_\Omega |\nabla u_k|^2 + (P^\perp u_k)^2\right),\right. \\ \left.\varphi\left(\int_\Omega |\nabla v_1|^2 + (P^\perp v_1)^2,\ldots, \int_\Omega |\nabla v_k|^2 + (P^\perp v_k)^2\right)\right)\\
+\frac{\beta}{q}\int_\Omega \Big(\sum_{j=1}^k u_i^2\Big)^q \Big(\sum_{i=1}^k v_i^2 \Big)^q.
\end{multline*}
\item[(c)] For $i,j = 1, \dots, k$ there exist Lagrange multipliers $\mu_{ij,\beta}, \nu_{ij,\beta} >0$, and coefficients
\begin{equation}\label{eq:coefficients}
\begin{aligned}
a_{i,\beta}&=\partial_1 F\left(\varphi(M(\mathbf{u})),\varphi(N(\mathbf{v}))\right)\cdot  \partial_{i} \varphi \left(\int_\Omega |\nabla u_{1}|^2+(P^\perp u_{1})^2,\ldots, \int_\Omega |\nabla  u_{k}|^2+(P^\perp u_{k})^2\right) > 0 \\
 b_{i,\beta}&=\partial_2 F\left(\varphi(M(\mathbf{u})),\varphi(N(\mathbf{v}))\right) \cdot \partial_{i} \varphi \left(\int_\Omega |\nabla v_{1}|^2+(Q^\perp v_{1})^2,\ldots, \int_\Omega |\nabla  v_{k}|^2+(Q^\perp  v_{k})^2\right) > 0
\end{aligned}
\end{equation}
such that the components of $(\mathbf{u}, \mathbf{v})$ solve the system
\begin{equation}\label{eq:equation_for_u_p_v_p}
\begin{cases}
a_{i,\beta}(-\Delta u_{i} + P^\perp u_i)=\sum_{j=1}^k \mu_{ij,\beta} u_{j}-\beta u_{i}\left(\sum_{j=1}^ku_{j}^2\right)^{q-1}\left(\sum_{j=1}^k v_{j}^2\right)^q  \\[10pt]
b_{i,\beta}(-\Delta v_{i}+ Q^\perp v_i)= \sum_{j=1}^k \nu_{ij,\beta} v_{j}-\beta v_{i} \left(\sum_{j=1}^k v_{j}^2\right)^{q-1} \left(\sum_{j=1}^k u_{j}^2\right)^{q} 
\end{cases} \quad \text{ in } \Omega.
\end{equation}
 
\end{enumerate}

\end{proposition}

In view of the previous result, whenever we refer to $\Mbeta$ we assume that its functions verify the additional conditions \eqref{eq:orthogonality} and \eqref{eq:monotonicity}.

\begin{proof}
The result follows by the critical point theory of functionals in Hilbert spaces. First, some preliminary remarks :
\begin{enumerate}
\item $\Sigma(L^2)$ is a $C^1$ submanifold of $H^1_0(\Omega,\R^k)$ of codimension $k(k+1)/2$ (see \cite[Lemma 3.7]{RamosTavaresTerracini}).
\item $E_{\beta} : H^1_0(\Omega) \times H^1_0(\Omega) \to \R^+$ is a $C^1$ functional and, for any $\boldsymbol{\xi},\boldsymbol{\eta} \in H^1_0(\Omega,\R^k)$, we have
\begin{align*}
	\frac{E_{\beta}'(\mathbf{u},\mathbf{v})(\boldsymbol{\xi},\boldsymbol{\eta})}{2} = &\partial_1 F\left(\varphi(M(\mathbf{u})),\varphi(N(\mathbf{v}))\right) \sum_{i\leq j}^k \frac{\partial}{\partial E_{ij}} \varphi(M(\mathbf{u}))\int_\Omega (\nabla u_i \cdot \nabla \xi_j +(P^\perp u_i) \xi_j)\\
					&+\partial_2 F\left(\varphi(M(\mathbf{u})),\varphi(N(\mathbf{v}))\right) \sum_{i\leq j}^k \frac{\partial }{\partial E_{ij}} \varphi(N(\mathbf{v}))\int_\Omega (\nabla v_i \cdot \nabla \eta_j + (Q^\perp v_j) \eta_j) \\
					&+\beta\sum_{i=1}^k \int_\Omega u_i\xi_i \Big(\sum_{j=1}^k u_i^2\Big)^{q-1}\Big( \sum_{i=1}^k v_i^2\Big)^q +\beta\sum_{i=1}^k \int_\Omega v_i\eta_i \Big(\sum_{j=1}^k v_i^2\Big)^{q-1}\Big( \sum_{i=1}^k u_i^2\Big)^q.
\end{align*}
\end{enumerate}
Let $\beta>0$. By Lemma \ref{lemma:easystatements} we have $c_\beta>-\infty$. We take a minimizing sequence $\mathbf{u}_n=(u_{1,n},\ldots, u_{k,n})$, $\mathbf{v}_n=(v_{1,n},\ldots, v_{k,n})\in \Sigma(L^2)$,  $E_{\beta}(\mathbf{u}_n,\mathbf{v}_n)\to c_{\beta}$ as $n\to \infty$.  By Ekeland's Variational Principle and by property (1) listed above, we can suppose without loss of generality that $E_{\beta}|_{\Sigma(L^2)}'(\mathbf{u}_n,\mathbf{v}_n)\to 0$ in $H^{-1}(\Omega,\R^k)$. For each $n \in \N$ take $O_{1,n},O_{2,n} \in \mathcal{O}_k(\R)$ such that $O_{1,n} M(\mathbf{u}_n) O_{1,n}^T$ and $O_{2,n} M(\mathbf{v}_n) O_{2,n}^T$ are diagonal matrices and let
\[
	\mathbf{\widetilde{u}}_n := O_{1,n} \mathbf{u}_n \quad \text{and} \quad \mathbf{\widetilde{v}}_n := O_{2,n} \mathbf{v}.
\]
Then $E_\beta(\mathbf{\widetilde{u}}_n, \mathbf{\widetilde{v}}_n)=E_\beta(\mathbf{u}_n, \mathbf{v}_n)$, $\mathbf{\widetilde{u}}_n, \mathbf{\widetilde{v}}_n \in \Sigma(L^2)$ and
\[
E_\beta(\mathbf{\widetilde u}_n,\mathbf{\widetilde v}_n)\to c_\beta,\qquad E_\beta'|_{\Sigma(L^2)}(\mathbf{\widetilde u}_n,\mathbf{\widetilde v}_n)\to 0 \quad \text{ as } n\to \infty.
\]
Therefore
\begin{multline*}
F\left( \varphi\left(\int_\Omega |\nabla \widetilde u_{1,n}|^2 + (P^\perp \widetilde u_{1,n})^2,\ldots, \int_\Omega |\nabla \widetilde u_{k,n}|^2 + (P^\perp \widetilde u_{k,n})^2\right),\right. \\ \left.\varphi\left(\int_\Omega |\nabla \widetilde v_{1,n}|^2 + (P^\perp \widetilde v_{1,n})^2,\ldots, \int_\Omega |\nabla \widetilde v_{k,n}|^2 + (P^\perp \widetilde v_{k,n})^2\right)\right) \leq E_{\beta}(	\mathbf{\widetilde{u}}_n,	\mathbf{\widetilde{v}}_n)\leq c_{\beta}+1
\end{multline*}
for large $n$. Since $\mathbf{\widetilde u}_n,\mathbf{\widetilde v}_n\in \Sigma(L^2)$ then 
\[
    \lambda_1(\Omega) \leq\int_\Omega | \nabla \widetilde u_{i,n}|^2 ,  \int_\Omega |\nabla \widetilde v_{i,n}|^2.
\]
Combining this information with (H1)--(H2) we deduce that $\mathbf{\widetilde u}_n,\mathbf{\widetilde v}_n$ are bounded sequences in $H^1_0(\Omega,\R^k)$, so that (up to subsequence) $\mathbf{\widetilde u}_n\rightharpoonup \mathbf{\widetilde u}$, $\mathbf{\widetilde v}_n\rightharpoonup \mathbf{\widetilde v}$ weakly in $H^1_0(\Omega,\R^k)$, strongly in $L^r(\Omega;\R^k)$, for every $1\leq r<2^*$. We can now conclude exactly as in \cite[Theorem 3.8]{RamosTavaresTerracini}, observing that $ \frac{\partial}{\partial E_{ij}} \varphi(M(\mathbf{\widetilde{u}}_n))= \frac{\partial}{\partial E_{ij}} \varphi(N(\mathbf{\widetilde{v}}_n))=0$ for $i\neq j$ (recall Lemma \ref{lemma:derivative_0_outsidediagonal}), that
\begin{align*}
& \partial_1 F\left(\varphi(M(\mathbf{\widetilde{u}}_n)),\varphi(N(\mathbf{\widetilde{v}}_n))\right) \partial_{i} \varphi \left(\int_\Omega |\nabla \widetilde u_{1,n}|^2+(P^\perp \widetilde u_{1,n})^2,\ldots, \int_\Omega |\nabla \widetilde u_{k,n}|^2+(P^\perp \widetilde u_{k,n})^2\right)\geq \delta>0,\\
& \partial_2 F\left(\varphi(M(\mathbf{\widetilde{v}}_n)),\varphi(N(\mathbf{\widetilde{v}}_n))\right) \partial_{i} \varphi \left(\int_\Omega |\nabla \widetilde v_{1,n}|^2+(P^\perp \widetilde v_{1,n})^2,\ldots, \int_\Omega |\nabla \widetilde v_{k,n}|^2+(P^\perp \widetilde v_{k,n})^2\right)\geq \delta>0
\end{align*}
for some $\delta>0$ independent from $n$, and that $\mathbf{\widetilde u}_n,\mathbf{\widetilde v}_n$ satisfy \eqref{eq:equation_for_u_p_v_p} up to an $\textrm{o}_n(1)$ perturbation in $H^{-1}(\Omega,\R^k)$. We can then conclude that actually $\mathbf{\widetilde u}_n, \mathbf{\widetilde v}_n$ converge strongly to $\mathbf{\widetilde u},\mathbf{\widetilde v}$ in $H^1_0(\Omega,\R^k)$, which solve \eqref{eq:equation_for_u_p_v_p}.
\end{proof}

\section{Asymptotic Limits: Proof of Theorem \ref{thm:main_result}}
We study the entirety of $\Mbeta$, the set of critical points of $E_\beta$ at level $c_\beta$, in order to establish its limit when $\beta \to + \infty$. Our main aim is to show that the functions in $\Mbeta$ are uniformly H\"older continuous in $\beta$.  This allows to prove strong convergence in $H^1$ to $(\boldsymbol{\phi},\boldsymbol{\psi})$, together with the desired regularity results in Theorem \ref{thm:main_result}.  

\subsection{Uniform bounds}

Recall the definition of $\widetilde c$ from \eqref{eq:OPPm=2_quasiopen}.  We start with some easier bounds of the $L^\infty$ and $H^1$ norms.

\begin{proposition}[Uniform $L^\infty$ and $H^1$ bounds]\label{unif h1 and linf}
We have 
\[
c_{\beta}\leq \widetilde c\qquad \text{ for every } \beta>0.
\]
There exists $C>0$ independent of $\beta$ such that for any $(\mathbf{u}_{\beta},\mathbf{v}_{\beta}) \in \Mbeta$ we have
\begin{equation*}
  \frac{\beta}{q}\int_\Omega \Big(\sum_{i=1}^k u_{i,\beta}^2\Big)^q \Big(\sum_{i=1}^k v_{i,\beta}^2\Big)^q\leq C
\end{equation*}
and
\begin{equation}\label{eq:unif h1 and linf_2}
\frac{1}{C}\leq a_{1,\beta}, \ldots, a_{k,\beta}\leq  C, \qquad  \frac{1}{C}\leq b_{1,\beta}, \ldots, b_{k,\beta}\leq  C.
\end{equation}
Furthermore,
\[
\| \mathbf{u}_{\beta}\|_{H^1_0(\Omega,\R^k)},\  \| \mathbf{v}_{\beta}\|_{H^1_0(\Omega,\R^k)}  \leq C ,\qquad \|  \mathbf{u}_{\beta}\|_{L^\infty(\Omega,\R^k)}, \|\mathbf{v}_{\beta}\|_{L^\infty(\Omega,\R^k)}\leq C.
\]
\end{proposition}

\begin{proof}
 Since $\phi_i=0$ a.e. in $\Omega \setminus \omega_1$, $\psi_i=0$ a.e. in $\Omega \setminus \omega_2$ and $|\omega_1\cap \omega_2|=0,$ then  $\phi_i \cdot \phi_j=0$ a.e. in $\Omega$ for every $i,j$, hence
\[
\int_\Omega \Big(\sum_{i=1}^k \phi_i^2\Big)^q \Big(\sum_{i=1}^k \psi_i^2 \Big)^q  =0.
\]
Moreover, $P^\perp \phi_i=Q^\perp \psi_i=0$, as $\phi_i\in L(\boldsymbol{\phi})$ and $\psi_i\in L(\boldsymbol{\psi})$. Therefore,  since $\boldsymbol{\phi},\boldsymbol{\psi}\in \Sigma(L^2)$, 
\begin{align*}
\widetilde c=& F\left(\varphi(\lambda_1(\omega_1),\ldots, \lambda_{k}(\omega_1)),\varphi(\lambda_1(\omega_2),\ldots, \lambda_{k}(\omega_2))\right)\\
	=& F\left( \varphi\left(\int_\Omega |\nabla \phi_1|^2,\ldots, \int_\Omega |\nabla \phi_k|^2\right)  ,  \varphi\left(\int_\Omega |\nabla \psi_1|^2,\ldots, \int_\Omega |\nabla \psi_k|^2\right)  \right) \\
						=& E_{\beta}(\boldsymbol{\phi},\boldsymbol{\psi})\geq  \min_{\mathbf{u},\mathbf{v}\in \Sigma(L^2)} E_{\beta}(\mathbf{u},\mathbf{v}) =c_{\beta}.
\end{align*}

By the monotonicity assumptions on $F$ and $\varphi$, and since $\beta>0$, we see that
\begin{multline*}
F\left( \varphi\left(\int_\Omega |\nabla u_{1,\beta}|^2,\ldots, \int_\Omega |\nabla u_{k,\beta}|^2\right)  ,  \varphi\left(\int_\Omega |\nabla v_{1,\beta}|^2,\ldots, \int_\Omega |\nabla v_{k,\beta}|^2\right)  \right)   \\
\leq	 F\left( \varphi\left(\int_\Omega |\nabla u_{1,\beta}|^2 + (P^\perp u_{1,\beta})^2,\ldots, \int_\Omega |\nabla u_{k,\beta}|^2 + (P^\perp u_{k,\beta})^2\right),\right.  \\
	 \left. \varphi\left(\int_\Omega |\nabla v_{1,\beta}|^2 + (P^\perp v_{1,\beta})^2,\ldots, \int_\Omega |\nabla v_{k,\beta}|^2 + (P^\perp v_{k,\beta})^2\right)\right) \\
	 + \frac{\beta}{q}\int_\Omega \Big(\sum_{i=1}^k u_{i,\beta}^2\Big)^q \Big(\sum_{i=1}^k v_{i,\beta}^2\Big)^q= E_{\beta}(\mathbf{u}_{\beta},\mathbf{v}_{\beta}) =  c_{\beta} \leq \widetilde c.  
	 \end{multline*}
Combining this with Lemma \ref{lemma:easystatements} and our assumptions of $F$ and $\varphi$, (H1)--(H2), we conclude that there exists a constant $C>0$ such that
\[
\int_\Omega |\nabla u_{i,\beta}|^2 + (P^\perp u_{i,\beta})^2, \ \int_\Omega |\nabla v_{i,\beta}|^2+(Q^\perp v_{i,\beta})^2,\  \frac{\beta}{q}\int_\Omega \Big(\sum_{i=1}^k u_{i,\beta}^2\Big)^q \Big(\sum_{i=1}^k v_{i,\beta}^2\Big)^q \leq C \qquad \text{ for all }  \beta>0.
\]

\smallbreak

Since $F$ and $\varphi$ are of class $C^1$, by \eqref{eq:coefficients} we conclude that $1/C\leq a_{i,\beta},b_{i,\beta}\leq C$ for some $C>0$.

The only thing left to prove is the $L^\infty$ uniform estimate. Let $i,l\in \{1,\ldots, k\}$. Testing the equation of $u_{i,\beta}$ in \eqref{eq:equation_for_u_p_v_p} by $u_{l,\beta}$ yields
\[
\mu_{il,\beta}=\delta_{il} a_{i,\beta}\int_\Omega (|\nabla u_{i,\beta}|^2+(P^\perp u_{i,\beta})^2) + \int_\Omega \beta u_{i,\beta}u_{l,\beta} \Big(\sum_{j=1}^k u_{j,\beta}^2\Big)^{q-1}\Big(\sum_{j=1}^k v_{j,\beta}^2\Big)^q
\]
and hence $|\mu_{il,p,\beta}|\leq C$ independently of $\beta>0$. Recall that $P^\perp u_{i,\beta}=u_{i,\beta}-\sum_{j=1}^k \langle u_{i,\beta},\phi_j \rangle_{L^2(\Omega)} \phi_j$. By Kato's inequality, we have
\begin{multline*}
-\Delta |u_{i,\beta}|\leq -\text{sign}(u_{i,\beta}) \Delta u_{i,\beta} \\
				 = \sum_{j=1}^k \frac{\mu_{ij,\beta}}{a_{i,\beta}} \text{sign}(u_{i,\beta}) u_{j,\beta} -  |u_{i,\beta}| + \sum_{j=1}^k \langle u_{i,\beta},\phi_j \rangle_{L^2(\Omega)} \text{sign}(u_{i,\beta}) \phi_j \\
				  - \beta|u_{i,\beta}| \Big(\sum_{j=1}^ku_{j}^2\Big)^{q-1}\Big(\sum_{j=1}^k v_{j}^2\Big)^q\\
				 \leq  \sum_{j=1}^k C |u_{j,\beta}| + \sum_{j=1}^k \langle u_{i,\beta},\phi_j \rangle_{L^2(\Omega)} \text{sign}(u_{i,\beta}) \phi_j .
\end{multline*}
By summing up for $i=1, \dots, k$ and letting $w_{\beta}:=\sum_{i=1}^k |u_{i,\beta}|\geq 0$, we have
\begin{equation}\label{eq:w_inequality_BrezisKato}
	-\Delta  w_{\beta} \leq C (w_{\beta} + \|w_{\beta}\|_{L^2(\Omega)}).
\end{equation}
Since $\{w_{\beta}\}$ is uniformly bounded in $L^2(\Omega)$, a Brezis-Kato type argument allows us to conclude. Indeed, assume that $w_{\beta}\in L^{2+\delta}(\Omega)$ for some $\delta\geq 0$. To simplify, we omit the dependent of $w$ on $\beta$ for the remainder of the proof, and consider $N\geq 3$ (otherwise the proof is simpler). Testing \eqref{eq:w_inequality_BrezisKato} by $w^{1+\delta}$, using Sobolev and H\"older inequalities, and denoting the best Sobolev constant of $H^1_0(\Omega)\hookrightarrow L^{2^*}(\Omega)$ by $C_S$ we find
\begin{multline*}
	C_S^2 \frac{1+\delta}{(1+\delta/2)^2} \|w\|_{L^{2^*(2+\delta)/2}(\Omega)}^{2+\delta} \leq  \frac{1+\delta}{(1+\delta/2)^2}\int_\Omega |\nabla w^{1+\delta/2}|^2\\
	\leq  C( \|w\|^{2+\delta}_{L^{2+\delta}(\Omega)} + \|w\|_{L^2(\Omega)}\|w\|_{L^{1+\delta}(\Omega)}) \leq  C \|w\|_{L^{2+\delta}(\Omega)}^{2+\delta}.
\end{multline*}
Hence there exists a constant $\kappa > 0$ such that
\[
	\|w\|_{L^{2^*(2+\delta)/2}(\Omega)} \leq \Big(\kappa \frac{(1+\delta/2)^2}{1+\delta}\Big)^\frac{1}{2+\delta} \|w\|_{L^{2+\delta}(\Omega)}.
\]
We wish to iterate this inequality in order to obtain a bound for the $L^\infty$ norm of $w$. To this end, let $\{\delta_n\}_n$  be the sequence of positive real numbers such that $\delta_0=0$ and $2+\delta_{n+1}= 2^*(2+\delta_n)/2$. We immediately note that $\delta_n\geq (2^*/2)^{n-1}$, thus
\[
	D:= \prod_{n=1}^\infty \left(\kappa \frac{(1+\delta_n/2)^2}{1+\delta_n}\right)^\frac{1}{2+\delta_n}=\exp\left(\sum_{n=1}^\infty \frac{\log \Big(\frac{\kappa(1+\delta_n/2)^2}{1+\delta_n}\Big)}{2+\delta_n}\right)<\infty.
\]
As a consequence
\[
	\|w\|_{L^\infty(\Omega)} \leq D \|w\|_{L^2(\Omega)}
\]
and the proof is concluded, as $w=w_{\beta}$ is uniformly bounded in $L^2(\Omega)$.
\end{proof}

\medbreak

We proceed our analysis of the family of solutions $\Mbeta$, focusing this time on stronger compactness results independent of the separation parameter $\beta > 0$. Our goal is to show that it is possible to take the limit as $\beta \to + \infty$ in the family of minimizers of Proposition \ref{thm:minimizer_for_c_p}. In particular, we want to apply the well-established framework of \cite{NTTV1, STTZ1, STZ1}. We start by some uniform estimates of the $C^{0,\alpha}$ norms of the solutions. Here we scheme through the proof of this result without entering too much into the details since the result, even though expected to hold, is not present in this from in the literature due to a different form of the competition term (cfr.\ in particular \cite{STTZ1}).

\begin{proposition}[Uniform H\"older bounds]\label{prop:uniform_Holder_bounds}
For any given $\alpha\in (0,1)$ there exists a constant $C_{\alpha}>0$, which may depend on $\alpha$ but not on $\beta$, such that for any $(\mathbf{u}_{\beta},\mathbf{v}_{\beta}) \in \Mbeta$ 
\[
	\|\mathbf{u}_{\beta}\|_{C^{0,\alpha}(\overline \Omega,\R^k)},\ \|\mathbf{v}_{\beta}\|_{C^{0,\alpha}(\overline \Omega,\R^k)}\leq C_{\alpha}.
\]
\end{proposition}

The proof is based on a contradiction argument, to which we dedicate the rest of this subsection. Let us assume that, for some $\alpha < 1$, there exists a sequence of solutions $(\mathbf{u}_{n}, \mathbf{v}_{n})$ whose $\alpha$-H\"older quotient is not bounded. Since the function $(\mathbf{u}_{\beta}, \mathbf{v}_{\beta})$ are smooth for $\beta$ bounded, it follows that necessarily $\beta_n \to + \infty$ and that there exists a sequence of points $(x_n, y_n) \in \bar \Omega \times \Omega$ such that
\begin{multline*}
    L_n := \max_{i,j = 1, \dots, k} \left\{ \max_{x,y \in \bar\Omega} \frac{|u_{i,n}(x)-u_{i,n}(y)|}{|x-y|^{\alpha}}, \max_{x,y \in \bar\Omega} \frac{|v_{i,n}y(x)-v_{i,n}(y)|}{|x-y|^{\alpha}} \right\} \\ = \max_{i,j = 1, \dots, k} \left\{ \frac{|u_{i,n}(x_n)-u_{i,n}(y_n)|}{|x_n-y_n|^{\alpha}}, \frac{|v_{i,n}(x_n)-v_{i,n}(y_n)|}{|x_n-y_n|^{\alpha}} \right\} \to \infty.
\end{multline*}
Letting $r_n = |x_n - y_n| \to 0$, we introduce a new family of functions, which are rescaled versions of $(\mathbf{u}_n,\mathbf{v}_n)$. Namely, for any $i = 1, \dots, k$, we let
\[
    \bar u_{i,n} := \frac{1}{L_n r_n^{\alpha}} u_{i,n}(x_n + r_n x), \qquad \bar v_{i,n} := \frac{1}{L_n r_n^{\alpha}} v_{i,n}(x_n + r_n x)
\]
for $x \in \Omega_n = \frac{\Omega - x_n}{r_n}$. From the definition, we observe that the functions $(\mathbf{\bar u}_{n}, \mathbf{\bar v}_{n})$, although they may not be uniformly bounded in $0$ for instance, they have uniformly bounded H\"older quotient of exponent $\alpha$ and moreover for each $n$ there exists a component in $(\mathbf{\bar u}_{n}, \mathbf{\bar v}_{n})$ whose oscillation in $B_1$ is equal to 1, that is
\begin{multline*}
    \max_{i,j = 1, \dots, k} \left\{ \max_{x,y \in \bar \Omega_n} \frac{|\bar u_{i,n}(x)-\bar u_{i,n}(y)|}{|x-y|^{\alpha}}, \max_{x,y \in \bar \Omega_n} \frac{|\bar v_{i,n}(x)-\bar v_{i,n}(y)|}{|x-y|^{\alpha}} \right\} \\ = \max_{i,j = 1, \dots, k} \left\{ \left|\bar u_{i,n}(0)-\bar u_{i,n}\left(\frac{y_n-x_n}{r_n}\right)\right| , \left|\bar v_{i,n}(0)-\bar v_{i,n}\left(\frac{y_n-x_n}{r_n}\right)\right| \right\} = 1.
\end{multline*}
Without loss of generality, we assume that
\begin{equation}\label{eq:oscillation=1}
     \left|\bar u_{1,n}(0)-\bar u_{1,n}\left(\frac{y_n-x_n}{r_n}\right)\right| = 1.
\end{equation}
Finally, a direct computation shows that $(\mathbf{\bar u}_{n}, \mathbf{\bar v}_{n})$ solves
\begin{equation}\label{syst_aux_n}
    \begin{cases}
    -a_{i,n}\Delta \bar u_{i,n} = \varepsilon_{i,n} - M_n \bar u_{i,n} \left(\sum_{j=1}^k \bar u_{j,n}^2\right)^{q-1}\left(\sum_{j=1}^k \bar v_{j,n}^2\right)^q \\
    -b_{i,n}\Delta \bar v_{i,n} = \delta_{i,n} - M_n \bar v_{i,n} \left(\sum_{j=1}^k \bar v_{j,n}^2\right)^{q-1} \left(\sum_{j=1}^k \bar u_{j,n}^2\right)^{q}
    \end{cases} \text{ in $\Omega_n$,}
\end{equation}
where the competition parameter is $M_n = \beta_n L_n^{4q-2} r_n^{2\alpha(2q-1)+2}$, and

\begin{equation}\label{sys bu}
    \begin{split}
        \varepsilon_{i,n}(x) &= r_n^{2-\alpha} L_n^{-1} \left(
-u_{i,\beta_n} + \sum_{j=1}^k \langle u_{i,\beta_n},\phi_j\rangle_{L^2(\Omega)} \phi_j +\sum_{j=1}^k \mu_{ij,\beta_n} u_{j,\beta_n}\right)(x_n + r_n x) \to 0\\
        \delta_{i,n}(x) &= r_n^{2-\alpha} L_n^{-1} \left( -v_{i,\beta_n} + \sum_{j=1}^k \langle v_{i,\beta_n},\psi_j\rangle_{L^2(\Omega)} \psi_j +\sum_{j=1}^k \nu_{ij,\beta_n} v_{j,\beta_n} \right)\left(x_n+r_n x\right) \to 0
    \end{split}
\end{equation}
uniformly in $\bar \Omega_n$ by Proposition \ref{unif h1 and linf} and since $\phi_j,\psi_j\in L^\infty(\Omega)$ for every $j$.

We now split the rest of the contradiction argument into several lemmas. 

\begin{lemma}\label{lem 0 bound}
The functions in $(\mathbf{\bar u}_{n}, \mathbf{\bar v}_{n})$ are uniformly locally bounded in $C^{0,\alpha}(\Omega_n)$. In particular, both
\[
    d_n := \sum_{i=1}^{k} \bar u_{i,n}^2(0) \qquad \text{and} \qquad e_n := \sum_{i=1}^{k} \bar v_{i,n}^2(0)
\]
are bounded uniformly.
\end{lemma}

We adapt the proof of \cite[Lemma 6.10]{STZ1} to our present context, which is based an a contradiction argument. We need an integral estimate on the size of the competition term. First of all we observe that if either $\{d_n\}$ or $\{e_n\}$ is unbounded, then necessarily $\Omega_n \to \R^n$ by the uniform estimate on the H\"older quotients of the blow-up sequence and since $\mathbf{u}_n=\mathbf{v}_n=0$ on $\partial \Omega_n$. In particular, we may assume that for any $x \in \R^n$ and $R > 0$, $B_R(x) \subset \Omega_n$ for any $n$ sufficiently large.

\begin{lemma}\label{lem int est}
Assume that either $d_n \to +\infty$ or $e_n \to +\infty$. For any $R > 0$ there exists $C(R) \geq 0$ such that for any $x \in \R^N$ and $n$ large enough
\[
	M_n \int_{B_R(x) } \left(\sum_{j=1}^k \bar u_{j,n}^2\right)^q \left(\sum_{j=1}^k \bar v_{j,n}^2\right)^q \leq C(R) \min\left(\sum_{j=1}^{k} \|\bar u_{j,n}\|_{L^\infty(B_{2R})}, \sum_{j=1}^{k} \|\bar v_{j,n}\|_{L^\infty(B_{2R})} \right).
\]
\end{lemma}
\begin{proof}
The proof follows verify closely the proof of \cite[Lemma 6.10]{STZ1}, thus we provide here and a sketch of it in the case $x = 0$. We consider the system \eqref{syst_aux_n}. Multiplying the equation in $\bar{u}_{j,n}$ by $\bar{u}_{j,n}$, integrating by parts in $B_R(0) $ and summing over $j$, we find
\[
	\begin{split}
  I(R) &:= \frac{1}{R^{N-2} } \int_{B_R }  a_{j,n}|\nabla \bar{u}_{j,n}|^2 - \sum_{j=1}^{k} \eps_{j,n}\bar{u}_{j,n} + M_n \left(\sum_{j=1}^k \bar u_{j,n}^2\right)^{q}\left(\sum_{j=1}^k \bar v_{j,n}^2\right)^q \, dx \\
  &= \frac{1}{R^{N-2}} \int_{\partial B_R } \sum_{j=1}^{k} \bar{u}_{j,n} \partial_{\nu} \bar{u}_{j,n} = \frac{1}{2R^{N-2}} \int_{\partial B_R } \partial_{\nu} \left(\sum_{j=1}^{k} \bar{u}_{j,n}^2 \right) = \frac{R}{2} \frac{d}{dR} \left( \frac{1}{R^{N-1}}\int_{\partial B_R }\sum_{j=1}^{k} \bar{u}_{j,n}^2\right).
  \end{split}
\] 
Exploiting the uniform H\"older bounds of the blow-up sequence we have
\begin{multline*}
    \int_{R}^{2R} \frac{2}{r} I(r) = \frac{1}{(2R)^{N-1}} \int_{\partial B_{2R} } \left(\sum_{j=1}^{k} \bar{u}_{j,n}^2\right) - \frac{1}{R^{N-1}} \int_{\partial B_R } \left(\sum_{j=1}^{k} \bar{u}_{j,n}^2\right) \\
     = \int_{\partial B_1} \sum_{j=1}^k \left( \bar u_{j,n}^2(2Rx) - \bar{u}_{j,n}^2(Rx) \right)
    = \int_{\partial B_1} \sum_{j=1}^k \left( \bar{u}_{j,n}(2Rx) - \bar{u}_{j,n}(Rx) \right)  \left( \bar{u}_{j,n}(2Rx) + \bar{u}_{j,n}(Rx) \right) \\
    \leq C(R) \left(\sum_{j=1}^k \|\bar{u}_{j,n}\|_{L^\infty(B_{2R})}\right).
\end{multline*}
On the other hand, taking also \eqref{eq:unif h1 and linf_2} into account, we can bound the same integral term from below as follows.
\[
\begin{split}
    \int_{R}^{2R} \frac{2}{r} I(r) &\geq \min_{s \in [R,2R]} I(s) \geq \frac{1}{R^{N-2}} \left( \frac{M_n}{C 2^{N-1}} \int_{B_R } \left(\sum_{j=1}^k \bar u_{j,n}^2\right)^{q}\left(\sum_{j=1}^k \bar v_{j,n}^2\right)^q - \int_{ B_{2R}}  \sum_{j=1}^{k} |\eps_{j,n}|  |\bar{u}_{j,n}| \right)\\
    &\geq C(R) \left(M_n \int_{B_R } \left(\sum_{j=1}^k \bar u_{j,n}^2\right)^{q}\left(\sum_{j=1}^k \bar v_{j,n}^2\right)^q  -  \max_{j=1,\dots,k} \|\eps_{j,n}\|_{L^{\infty}} \left( \sum_{j=1}^{k}\|\bar u_{j,n}\|_{L^\infty(B_{2R})}\right)\right).
\end{split}
\]
We can reach an analogous conclusion by taking into account the equations satisfied by $\bar{\mathbf{v}}_n$. The conclusion follows by joining the two estimates together with \eqref{sys bu}.
\end{proof}

\begin{proof}[Proof of Lemma \ref{lem 0 bound}] 
To prove the result we argue by contradiction, excluding different possibilities for the sequences $\{d_n\}$ and $\{e_n\}$. Specifically we show that the assumption that the one of these two sequences is unbounded is incompatible with the uniform H\"older bounds of 
the blow-up sequence.

\noindent\textbf{Case 1.} We start by excluding the case in which both sequences $d_n$ and $e_n$ are unbounded. Exploiting the uniform bounds of the $C^{0,\alpha}$-seminorm of $\bar{\mathbf{u}}_n$ and $\bar{\mathbf{v}}_n$ we find from Lemma \ref{lem int est} that for some $R>0$ there exists $\bar n$ such that if $n\geq \bar n$ then
\begin{multline*}
	\frac14 M_n \sum_{j=1}^{k} |\bar u_{j,n}|(0) \left(\sum_{j=1}^k \bar u_{j,n}^2(0)\right)^{q-1} \left(\sum_{j=1}^k \bar v_{j,n}^2(0)\right)^q 
	\leq \frac12 M_n \frac{ \left(\sum_{j=1}^k \bar u_{j,n}^2(0)\right)^q \left(\sum_{j=1}^k \bar v_{j,n}^2(0)\right)^q }{ \left( 1 + \sum_{j=1}^{k} |\bar u_{j,n}|(0) \right) } \\
	\leq M_n \frac{ \left(\sum_{j=1}^k ( \bar u_{j,n}(0) - R^\alpha)^2\right)^q \left(\sum_{j=1}^k (\bar v_{j,n}(0)-R^\alpha)^2\right)^q }{ \left( 1 + \sum_{j=1}^{k} |\bar u_{j,n}|(0) \right) } \\ \times \frac12 \left(\frac{ \sum_{j=1}^k \bar u_{j,n}^2(0)}{\sum_{j=1}^k ( \bar u_{j,n}(0) - R^\alpha)^2} \right)^q \left( \frac{ \sum_{j=1}^k \bar v_{j,n}^2(0) }{ \sum_{j=1}^k (\bar v_{j,n}(0)-R^\alpha)^2} \right)^q \\
	\leq  M_n \frac{ \int_{B_R(0) } \left(\sum_{j=1}^k \bar u_{j,n}^2\right)^q 
\left(\sum_{j=1}^k \bar v_{j,n}^2\right)^q }{ |B_R(0) | \left( 1 + \sum_{j=1}^{k} |\bar u_{j,n}|(0) \right) } \leq C(R).
\end{multline*}
In particular, since $d_n,e_n\to +\infty$, we obtain that in this case $M_n \to 0$. Moreover there exists $\Lambda \in \R$ such that
\[
	M_n \bar u_{1,n}(x) \left(\sum_{j=1}^k \bar u_{j,n}^2(x)\right)^{q-1} \left(\sum_{j=1}^k \bar v_{j,n}^2(x)\right)^q \to \Lambda
\]
uniformly in any compact set of $\Omega_n$. Indeed for any $K \subset \R^n$
\[
    \begin{split}
        M_n &\sup_{y \in K }\left| \bar u_{1,n}(0) \left(\sum_{j=1}^k \bar u_{j,n}^2(0)\right)^{q-1} \left(\sum_{j=1}^k \bar v_{j,n}^2(0)\right)^q - \bar u_{1,n}(y) \left(\sum_{j=1}^k \bar u_{j,n}^2(y)\right)^{q-1} \left(\sum_{j=1}^k \bar v_{j,n}^2(y)\right)^q\right| \\
        \leq &M_n \sup_{y \in K }\left| \bar u_{1,n}(0) - \bar u_{1,n}(y) \right| \left(\sum_{j=1}^k \bar u_{j,n}^2(0)\right)^{q-1} \left(\sum_{j=1}^k \bar v_{j,n}^2(0)\right)^q \\
        &+ M_n \sup_{y \in K }| \bar u_{1,n}(y)| \left| \left(\sum_{j=1}^k \bar u_{j,n}^2(0)\right)^{q-1} - \left(\sum_{j=1}^k \bar u_{j,n}^2(y)\right)^{q-1} \right| \left(\sum_{j=1}^k \bar v_{j,n}^2(0)\right)^q \\
    &+ M_n \sup_{y \in K } |\bar u_{1,n}(y)| \left(\sum_{j=1}^k \bar u_{j,n}^2(y)\right)^{q-1} \left| \left(\sum_{j=1}^k \bar v_{j,n}^2(0)\right)^q - \left(\sum_{j=1}^k \bar v_{j,n}^2(y)\right)^q\right| 
        \end{split}
    \]
    \[
    \begin{split}
    \phantom{M_n} \leq &M_n \sup_{y \in K }\left| 1 - \frac{\bar u_{1,n}(y)}{\bar u_{1,n}(0)} \right| \left| \bar u_{1,n}(0)\right| \left(\sum_{j=1}^k \bar u_{j,n}^2(0)\right)^{q-1} \left(\sum_{j=1}^k \bar v_{j,n}^2(0)\right)^q \\
        &+ M_n \sup_{y \in K } \left| 1 - \frac{\left(\sum_{j=1}^k \bar u_{j,n}^2(y)\right)^{q-1}}{\left(\sum_{j=1}^k \bar u_{j,n}^2(0)\right)^{q-1}} \right| \left| \frac{\bar u_{1,n}(y)}{\bar u_{1,n}(0)} \right|  | \bar u_{1,n}(0)| \left(\sum_{j=1}^k \bar u_{j,n}^2(0)\right)^{q-1}  \left(\sum_{j=1}^k \bar v_{j,n}^2(0)\right)^q \\
    &+ M_n \sup_{y \in K } |\bar u_{1,n}(0)| \left(\sum_{j=1}^k \bar u_{j,n}^2(0)\right)^{q-1} \left(\sum_{j=1}^k \bar v_{j,n}^2(0)\right)^q \\ 
    &\hspace{5cm}\times \left| 1 - \frac{\left(\sum_{j=1}^k \bar v_{j,n}^2(y)\right)^q}{\left(\sum_{j=1}^k \bar v_{j,n}^2(0)\right)^q}\right| \left| \frac{\left(\sum_{j=1}^k \bar u_{j,n}^2(y)\right)^{q-1}}{\left(\sum_{j=1}^k \bar u_{j,n}^2(0)\right)^{q-1}} \right| \left| \frac{\bar u_{1,n}(y)}{\bar u_{1,n}(0)} \right| \\
    \leq  & C(R) \sup_{y \in K }\left| 1 - \frac{\bar u_{1,n}(y)}{\bar u_{1,n}(0)} \right| + C(R) \sup_{y \in K } \left| 1 - \frac{\left(\sum_{j=1}^k \bar u_{j,n}^2(y)\right)^{q-1}}{\left(\sum_{j=1}^k \bar u_{j,n}^2(0)\right)^{q-1}} \right| \left| \frac{\bar u_{1,n}(y)}{\bar u_{1,n}(0)} \right| \\
    &+C(R)\sup_{y \in K }\left| 1 - \frac{\left(\sum_{j=1}^k \bar v_{j,n}^2(y)\right)^q}{\left(\sum_{j=1}^k \bar v_{j,n}^2(0)\right)^q}\right| \left| \frac{\left(\sum_{j=1}^k \bar u_{j,n}^2(y)\right)^{q-1}}{\left(\sum_{j=1}^k \bar u_{j,n}^2(0)\right)^{q-1}} \right| \left| \frac{\bar u_{1,n}(y)}{\bar u_{1,n}(0)} \right|  \to 0.
    \end{split}
\]
We introduce now an auxiliary sequence of functions by letting $w_n := \bar u_{1,n} - \bar u_{1,n}(0)$. The sequence $\{w_n\}$ is uniformly bounded in $C^{0,\alpha}_{\loc}$ and, up to striking out a subsequence, there exists $w \in C^{0,\alpha}_\loc(\R^n)$ such that $w_n \to w$ locally uniformly (and in $C^{0,\gamma}_\loc(\R^n)$ for any $\gamma \in (0,\alpha)$), $w$ is globally H\"older continuous of exponent $\alpha<1$, $w$ is not constant and it solves the equation (for $a_i:=\lim a_{1,n})$
\[
	-a_1\Delta w = -\Lambda \qquad \text{in $\R^n$},
\]
a contradiction. Indeed $w = h + \Lambda/(2n) |x|^2$ where $h$ is harmonic which grows at most quadratically (since $|h(x)| \leq  \Lambda/(2n) |x|^2 + |w(x)|$), thus $h$ is a harmonic polynomial of degree at most $2$, but since $w$ is globally H\"older continuous this implies that $h(x) \sim - \Lambda/(2n) |x|^2$ for $|x|\to +\infty$, which is impossible.

\noindent\textbf{Case 2.} We exclude the case in which the sequence $\{d_n\}$ is bounded while $\{e_n\}$ is unbounded. Observe that, in this case, the sequence $\{\bar{\mathbf{u}}_{n}\}$ is uniformly bounded in $C^{0,\alpha}_{\loc}$ and, up to striking out a subsequence, there exists a vector $\mathbf{w} \in C^{0,\alpha}(\R^n)$ such that $\bar{\mathbf{u}}_n \to \mathbf{w}$ locally uniformly, $\mathbf{w}$ is globally H\"older continuous of exponent $\alpha$, at least its first component $w_1$ is not constant by \eqref{eq:oscillation=1}. Since at least $w_1$ is not identically $0$ in $B_1$, we can again exploit Lemma \ref{lem int est} in order to conclude that there exist $R>0$ small and constants $C, C'>0$ such that
\[
	\begin{split}
	M_n \left(\sum_{j=1}^k \bar v_{j,n}^2(0)\right)^q &= M_n \left(\sum_{j=1}^k \bar v_{j,n}^2(0)\right)^q \frac{ M_n \int_{B_R(x) } \left(\sum_{j=1}^k \bar u_{j,n}^2\right)^q \left(\sum_{j=1}^k \bar v_{j,n}^2\right)^q  }{ M_n \int_{B_R(x) } \left(\sum_{j=1}^k \bar u_{j,n}^2\right)^q \left(\sum_{j=1}^k \bar v_{j,n}^2\right)^q } \\
	&= \frac{ M_n \int_{B_R(x) } \left(\sum_{j=1}^k \bar u_{j,n}^2\right)^q \left(\sum_{j=1}^k \bar v_{j,n}^2\right)^q  }{\int_{B_R(x) } \left(\sum_{j=1}^k \bar u_{j,n}^2\right)^q \frac{ \left(\sum_{j=1}^k \bar v_{j,n}^2\right)^q}{ \left(\sum_{j=1}^k \bar v_{j,n}^2(0)\right)^q } } \leq 2 \frac{ C }{ \int_{B_R(x) } \left(\bar u_{1,n}^2\right)^q } \leq C'.
	\end{split}
\]
Thus $M_n \to 0$ bounded and there exists a constant $\Lambda \geq 0$ such that
\[
	M_n \left(\sum_{j=1}^k \bar v_{j,n}^2(x)\right)^q \to \Lambda
\]
uniformly on compact subsets of $\R^n$. We conclude that $\mathbf{w}$ has at least one component (its first one) not constant and it solves
\[
	-a_i\Delta w_{i} = - \Lambda w_{i}(x) \left(\sum_{j=1}^k w_{j}^2(x)\right)^{q-1}
\]
a contradiction by applying  \cite[Lemma A.3]{STTZ1} to $|w_i|$.

\noindent\textbf{Case 3.} Similarly, we now exclude the possibility $\{d_n\}$ is unbounded, $\{e_n\}$ is bounded and there exists $x \in \R^n$ and $C$ such that $e_n(x) \geq C > 0$. Indeed, as in the previous case we find that there exists $C>0$ such that
\[
	M_n \left(\sum_{j=1}^k \bar u_{j,n}^2(x)\right)^q \leq C
\]
thus $M_n \to 0$ and there exists $\Lambda$
\[
    M_n \left(\sum_{j=1}^k \bar u_{j,n}^2\right)^q \to \Lambda.
\]
Then, by assumption the sequence $\{\bar{\mathbf{v}}_{n}\}$ is uniformly bounded in $C^{0,\alpha}_{\loc}$ and, up to striking out a subsequence, there exists a vector $\mathbf{z} \in C^{0,\alpha}(\R^n)$ such that $\bar{\mathbf{v}}_n \to \mathbf{z}$ locally uniformly, $\mathbf{z}$ is globally H\"older continuous of exponent $\alpha$, at least one component of $\mathbf{z}$ is not zero and it solves
\[
	-b_i\Delta z_{i} = - \Lambda z_{i}(x) \left(\sum_{j=1}^k z_{j}^2(x)\right)^{q-1}
\]
which implies that $\Lambda = 0$ (and $\mathbf{z}$ constant). But then letting $w_n := \bar u_{1,n} - \bar u_{1,n}(0)$, then $\{w_n\}$ is uniformly bounded in $C^{0,\alpha}_{\loc}$ and, up to striking out a subsequence, there exists $w \in C^{0,\alpha}_\loc(\R^n)$ such that $w_n \to w$ locally uniformly, $w$ is globally H\"older continuous of exponent $\alpha<1$, $w$ is not constant and it solves
\[
	-\Delta w = 0
\]
in contradiction with the classical theorem by Liouville on entire harmonic functions.

\noindent\textbf{Case 4.} Thus we need to exclude the case $\{d_n\}$ is unbounded but $\{e_n\}$ is bounded and $e_n(x) \to 0$ locally uniformly. Again by Lemma \ref{lem int est} we find that for any $x \in \Omega_n$ and $R>0$ we have
\[
	M_n \int_{B_R } \left(\sum_{j=1}^k \bar u_{j,n}^2\right)^{q}\left(\sum_{j=1}^k \bar v_{j,n}^2\right)^q \leq C(R) \sum_{j=1}^{k} \|\bar v_{j,n}\|_{L^\infty(B_{2R})} \to 0.
\]
Let $\eta \in C^{\infty}_0(\R^n)$ be any test function. By multiplying the equation in $\bar{u}_{1,n}$ by $\eta$ and integrating by parts we find
\begin{multline*}
 	\int \nabla \bar{u}_{1,n} \nabla \eta = \int \eps_{1,n} \eta -  M_n \bar u_{j,n} \eta \left(\sum_{j=1}^k \bar u_{j,n}^2\right)^{q-1}\left(\sum_{j=1}^k \bar v_{j,n}^2\right)^q \\
 	\leq \|\eps_{1,n}\|_{L^\infty} \|\eta \|_{L^1(\R^n)} \eta +  M_n \|\eta\|_{L^\infty(\R^n)} \int_{B_R} \left(\sum_{j=1}^k \bar u_{j,n}^2\right)^{q}\left(\sum_{j=1}^k \bar v_{j,n}^2\right)^q \to 0
\end{multline*}
for any $R>0$ such that $\supp \eta \subset B_R$. Letting once more $w_n := \bar u_{1,n} - \bar u_{1,n}(0)$, the sequence $\{w_n\}$ is uniformly bounded in $C^{0,\alpha}_{\loc}$ and, up to striking out a subsequence, there exists $w \in C^{0,\alpha}_\loc(\R^n)$ such that $w_n \to w$ locally uniformly (and in $C^{0,\gamma}_\loc(\R^n)$ for any $\gamma \in (0,\alpha)$), $w$ is globally H\"older continuous of exponent $\alpha$, $w$ is not constant and it solves the equation
\[
		-\Delta w = 0 \qquad \text{in $\R^n$}
\]
a contradiction.
\end{proof}

As a consequence of the previous result, we have that, up to striking out a subsequence, the sequence $\{ (\mathbf{\bar u}_n, \mathbf{\bar v}_n) \}_{n \in \N}$ converges in $C^{0,\gamma}_\loc$ for any $\gamma < \alpha$ to some limiting entire profile $(\mathbf{\bar u}, \mathbf{\bar v}) \in C^{0,\alpha}$. Reasoning as in \cite[pp. 293--294]{NTTV1} we have the following.
\begin{lemma}
The convergence of (a subsequence of) $(\mathbf{\bar u}_n, \mathbf{\bar v}_n)$ to its limit $(\mathbf{\bar u}, \mathbf{\bar v})$ is also strong in $H^1_\loc(\R^N)$.
\end{lemma}

In order to conclude, we have to analyze the following three possible case: $M_n \to 0$, $M_n$ bounded and $M_n \to \infty$.

\begin{lemma}
There exists $C>0$ such that $M_n \geq C$ for all $n$.
\end{lemma}
\begin{proof}
Indeed, assume by contradiction that there exists a subsequence in $(\mathbf{\bar u}_n, \mathbf{\bar v}_n)$ for which $M_n \to 0$. Then, from the local uniform convergence of  $(\mathbf{\bar u}_n, \mathbf{\bar v}_n)$ we obtain that the limit $(\mathbf{\bar u}, \mathbf{\bar v})$ is made of entire harmonic functions with bounded $C^{0,\alpha}$ semi-norm. Consequently they all must be constant, in contrast with the limit of the oscillation in $B_1$ of the first component.
\end{proof}

\begin{lemma}
It must be that $\lim_n M_n = +\infty$.
\end{lemma}
\begin{proof}
We may reason as before, assuming that $M_n \to 1$. We then end up with limiting functions $(\mathbf{\bar u}, \mathbf{\bar v})$ which solve
\[
    \begin{cases}
    -a_{i}\Delta \bar u_{i}=  - \bar u_{i} \left(\sum_{j=1}^k \bar u_{j}^2\right)^{q-1} \left(\sum_{j=1}^k \bar v_{j}^2\right)^q \\
    -b_{i}\Delta \bar v_{i}=  - \bar v_{i} \left(\sum_{j=1}^k \bar v_{j}^2\right)^{q-1} \left(\sum_{j=1}^k \bar u_{j}^2\right)^q
    \end{cases} \text{ in $\R^N$,}
\]
and the conclusion follows as in \cite[Claim 2. pag 18]{RamosTavaresTerracini}.
\end{proof}

Finally, let us address the case $M_n \to \infty$. In this case, in order to find a contradiction, we need to ensure the validity of an Almgren-type monotonicity formula for the limit profiles $(\mathbf{\bar u}, \mathbf{\bar v})$. To this end, we let first show the following

\begin{lemma}
For any $x \in \R^N$ and almost every $r > 0$, the following identity holds
\begin{multline*}
    (2 - N) \int_{B_r(x_0)} \sum_{i=1}^{k} \left( a_i |\nabla \bar u_i|^2  + b_i |\nabla \bar v_i|^2\right) + r \int_{\partial B_r(x_0)} \sum_{i=1}^{k}\left( a_i |\nabla \bar u_i|^2  + b_i |\nabla \bar v_i|^2 \right) \\ = 2 r \int_{\partial B_r(x_0)} \sum_{i=1}^{k} \left( a_i (\partial_\nu \bar u_i)^2  + b_i (\partial_\nu \bar v_i)^2 \right).
\end{multline*}
\end{lemma}
\begin{proof}
The proof follows mainly by a direct computation. For easier notation, let us consider the case $x_0 = 0$. Testing the equation in $(\mathbf{\bar u}_{n}, \mathbf{\bar v}_{n})$ by $(x\cdot \nabla \mathbf{\bar u}_{n}, x\cdot \nabla \mathbf{\bar v}_{n})$ and summing over $i = 1, \dots, k$, we obtain integrating by parts
\begin{multline*}
    \int_{B_r}\sum_{i=1}^{k}\left(-a_{i,n}\Delta \bar u_{i,n}  x \cdot \nabla \bar u_{i,n} - b_{i,n}\Delta \bar v_{i,n} x \cdot \nabla \bar v_{i,n}  \right) \\
     = \left(1-\frac{N}{2}\right) \int_{B_r} \sum_{i=1}^{k} \left( a_{i,n} |\nabla \bar u_{i,n}|^2  + b_i |\nabla \bar v_{i,n}|^2\right)+ \frac{r}{2} \int_{\partial B_r} \sum_{i=1}^{k}\left( a_{i,n} |\nabla \bar u_{i,n}|^2  + b_{i,n} |\nabla \bar v_{i,n}|^2 \right)   \\ - r \int_{\partial B_r} \sum_{i=1}^{k} \left( a_{i,n} (\partial_\nu \bar u_{i,n})^2  + b_{i,n} (\partial_\nu \bar v_{i,n})^2 \right).
\end{multline*}
We observe that, due to the strong $H^1$ convergence, the right hand side of the previous expression passes to the limit for almost every radius $r>0$.
On the other hand, replacing the equation in the left hand side, we find
\begin{multline}\label{eqn pohozaev n}
    M_n\int_{B_r} \sum_{i=1}^{k} \left( \bar u_{i,n}  x\cdot \nabla \bar u_{i,n} \left(\sum_{j=1}^k \bar u_{j,n}^2\right)^{\hspace{-.2em}q-1}\left(\sum_{j=1}^k \bar v_{j,n}^2\right)^{\hspace{-.2em}q} + \bar v_{i,n}  x\cdot \nabla \bar v_{i,n} \left(\sum_{j=1}^k \bar v_{j,n}^2\right)^{\hspace{-.2em}q-1}\left(\sum_{j=1}^k \bar u_{j,n}^2\right)^{\hspace{-.2em}q}\right) \\
    = M_n \frac{1}{2q} \int_{B_r} x \cdot \nabla \left(\sum_{j=1}^k \bar u_{j,n}^2\right)^{q}\left(\sum_{j=1}^k \bar v_{j,n}^2\right)^q \\
    = M_n \frac{N}{4q} \int_{B_r} \left(\sum_{j=1}^k \bar u_{j,n}^2\right)^{q}\left(\sum_{j=1}^k \bar v_{j,n}^2\right)^q - M_n \frac{r}{4q}\int_{\partial B_r} \left(\sum_{j=1}^k \bar u_{j,n}^2\right)^{q}\left(\sum_{j=1}^k \bar v_{j,n}^2\right)^q.
\end{multline}

We now go back to the equations in $(\mathbf{\bar u}_{n}, \mathbf{\bar v}_{n})$. By Kato's inequality we find that there exists a positive constant $C$, independent of $n$, such that
\[
	-\Delta |\bar u_{i,n}|+ M_n |\bar u_{i,n}| \left(\sum_{j=1}^k \bar u_{j,n}^2\right)^{q-1}\left(\sum_{j=1}^k \bar v_{j,n}^2\right)^{q}\leq C
\]
and similarly for $\bar v_{i,n}$. Let $r > 0$ be any fixed radius, we multiply the previous inequality by a smooth cut-function $\eta \in C_0^\infty(B_{3r})$  such that 
\[
	 \begin{cases}
		\eta(x) = 1 &\text{if $|x| \leq r$}\\
		\eta(x) \in (0,1) &\text{if $r < |x| < 3r$}
	\end{cases}, \quad \|\nabla \eta \|_{L^\infty} \leq 1/r.
\]
Integrating by parts yields the estimate
\[
	M_n \int_{B_r} |\bar u_{i,n}| \left(\sum_{j=1}^k \bar u_{j,n}^2\right)^{q-1}\left(\sum_{j=1}^k \bar v_{j,n}^2\right)^{q}, M_n\int_{B_r} |\bar v_{i,n}| \left(\sum_{j=1}^k \bar u_{j,n}^2\right)^{q}\left(\sum_{j=1}^k \bar v_{j,n}^2\right)^{q-1}\leq C(r).
\]
We obtain that
\[
    \lim_{n\to \infty} \int_{B_r}M_n \left(\sum_{j=1}^k \bar u_{j,n}^2\right)^{q}\left(\sum_{j=1}^k \bar v_{j,n}^2\right)^q = 0 \quad \text{for any $r >0$}
\]
and thus, by Fubini's theorem, for almost every radius $r > 0$ the right hand side in \eqref{eqn pohozaev n} vanishes as $n \to +\infty$. Finally, we observe that thanks to the $H^1$ converge of $(\mathbf{\bar u}_n, \mathbf{\bar v}_n)$ and the uniform vanishing of $(\varepsilon_{n}, \delta_n)$ (see eq.\ \ref{sys bu}), we have
\[
    \lim_{n \to +\infty} \int_{B_r} \sum_{i=1}^{k} \left(\varepsilon_{i,n}x \cdot  \nabla \bar u_{i,n} +\delta_{i,n}x \cdot  \nabla \bar v_{i,n}\right) = 0
\]
for every radius $r > 0$. The proof follows recollecting the previous observations.
\end{proof}

We are in position to conclude the uniform regularity result.
\begin{proof}[Proof of Proposition \ref{prop:uniform_Holder_bounds}]
As of now, we have obtained that, if there is no uniform H\"older bound, then necessarily $M_n \to \infty$. From this point on, the conclusion follows exactly as in \cite[Step B. page 19]{RamosTavaresTerracini}.
\end{proof}

\subsection{Conclusion of the proof of Theorem \ref{thm:main_result}}

From the previous results we can completely characterize the limit profiles as $\beta\to \infty$. 

\begin{proposition}[Limit as $\beta\to \infty$]\label{prop_limit_in_p}
Let $(\mathbf{u}_{\beta},\mathbf{v}_{\beta}) \in \Mbeta$. Then
\begin{equation}\label{eq:interactionterm_to0}
\lim_{\beta\to +\infty} \frac{\beta}{q}\int_\Omega \Big(\sum_{j=1}^k u_{j,\beta}^2\Big)^q \Big(\sum_{i=1}^k v_{j,\beta}^2 \Big)^q=0.
\end{equation}
Moreover, there exist $\mathbf{u}=(u_1,\ldots, u_k),\mathbf{v}=(v_1,\ldots, v_k)\in C^{0,\alpha}(\overline{\Omega}; \R^k)\cap H^1_0(\Omega,\R^k)$ such that, up to subsequence:
\begin{enumerate}
\item $\mathbf{u}_{\beta}\to \mathbf{u}$, $\mathbf{v}_{\beta}\to \mathbf{v}$  as $\beta\to +\infty$, strongly in $ H^1_0(\Omega,\R^k)$ and in $C^{0,\alpha}(\overline{\Omega},\R^k)$ for every $\alpha\in (0,1)$.
\item $u_{i}\cdot v_{j}= 0$ in $\Omega$ for every $i,j=1,\ldots, k$, and 
\[
(O_1,O_2):=\left( \left\{|\mathbf{u}|>0\right\},\left\{|\mathbf{v}|>0\right\}\right)\in \mathcal{P}_2(\Omega);
\]
\item $\mathbf{u},\mathbf{v}\in \Sigma(L^2)$;
\item we have
\begin{align*}
    \int_\Omega \nabla u_{i}\cdot \nabla u_{j} + (P^\perp u_i) (P^\perp u_j) &=\int_\Omega \nabla v_{i}\cdot \nabla v_{j} + (Q^\perp v_i) (Q^\perp v_j) =0& \forall i\neq j  \\
    \int_\Omega |\nabla u_{i}|^2 + (P^\perp {u}_i)^2  \leq \int_\Omega |\nabla u_{j}|^2 + (P^\perp u_j)^2 &, \; \int_\Omega |\nabla v_{i}|^2 + (Q^\perp v_i)^2 \leq \int_\Omega |\nabla v_{j}|^2 + (Q^\perp v_j)^2 & \forall i\leq j. 
\end{align*}
\end{enumerate}
As a consequence we have
\[
	\lim_{\beta \to +\infty} E_{\beta}(\mathbf{u}_\beta, \mathbf{v}_\beta) = \widetilde c.
\]
\end{proposition}

\begin{proof}
We only sketch the proof of these results, referring to \cite[p. 294]{NTTV1} for a complete and detailed proof. Recall the uniform bounds in Propositions \ref{unif h1 and linf} and \ref{prop:uniform_Holder_bounds}. Since $C^{0,\alpha}(\overline \Omega) \hookrightarrow C^{0,\gamma}(\overline \Omega)$ is a compact embedding whenever $0<\gamma<\alpha<1$, we have (up to a subsequence)
\begin{equation}\label{eq:convergence_in_beta_aux}
\mathbf{u}_{\beta}\to \mathbf{u},\quad \mathbf{v}_{\beta}\to \mathbf{v} \qquad \text{ as } \beta\to \infty,
\end{equation}
weakly in $H^1_0(\Omega,\R^k)$ and strongly in $C^{0,\alpha}(\overline \Omega,\R^k)\cap L^p(\Omega)$ for every $\alpha \in (0,1)$, $p\in [1,+\infty]$. By combining this information with Proposition \ref{unif h1 and linf} we have items (2) and (3). By Kato's inequality and the bounds mentioned before, we have the existence of $C>0$ independent on $\beta$ such that
\[
-\Delta |u_{i,\beta}|+ \beta |u_{i,\beta}| \Big(\sum_{j=1}^k u_{j,\beta}^2\Big)^{q-1}\Big(\sum_{j=1}^k v_{j,\beta}^2\Big)^{q}\leq C,
\]
and the same holds for the equation of $v_{i,\beta}$. Since $\Omega$ is smooth $\partial_\nu |u_{i,\beta}|,\partial_\nu |v_{i,\beta}|\leq 0$ on $\partial \Omega$ and an integration of the previous differential inequality yields
\[
\beta\int_\Omega |u_{i,\beta}| \Big(\sum_{j=1}^k u_{j,\beta}^2\Big)^{q-1}\Big(\sum_{j=1}^k v_{j,\beta}^2\Big)^{q}, \beta\int_\Omega |v_{i,\beta}| \Big(\sum_{j=1}^k u_{j,\beta}^2\Big)^{q}\Big(\sum_{j=1}^k v_{j,\beta}^2\Big)^{q-1}\leq C.
\]
We can deduce \eqref{eq:interactionterm_to0}. Moreover, testing the equation of $u_{i,\beta}$ with $u_{i,\beta}-u_{i}$ and the one of $v_{i,\beta}$ with $v_{i,\beta}-v_{i}$  implies that in \eqref{eq:convergence_in_beta_aux} the $H^1_0$--convergence is actually strong, so that (1) is proved. Finally, (4) is a direct consequence of this strong convergence combined with \eqref{eq:orthogonality}--\eqref{eq:monotonicity}
\end{proof}

\begin{proposition}\label{prop:limit_in_p}
From the family of functions $(\mathbf{u}_{\beta},\mathbf{v}_{\beta})$ in Proposition \ref{thm:minimizer_for_c_p} we consider any converging subsequence, and let $(\mathbf{u},\mathbf{v}):=\lim_{\beta\to \infty} (\mathbf{u}_{\beta},\mathbf{v}_{\beta})$ be any limit profile, as in the previous lemma. Then:
\begin{enumerate}
\item regarding the parameters, we have:
\begin{equation}\label{eq:limits_a_b}
    \begin{gathered} 
    \lim_\beta \mu_{ii,\beta}=:\mu_{ii}>0,\quad  \lim_{\beta} \nu_{ii,\beta}=:\nu_{ii}>0,  \quad\lim_{\beta} \mu_{ij,\beta}=\lim_{\beta} \nu_{ij,\beta}=0 \text{ for $i\neq j$},\\
    \lim_{\beta} a_{i,\beta}=:a_{i}>0,\quad \lim_{\beta} b_{i,\beta}=:b_{i}>0,
    \end{gathered}
\end{equation}
\item the limit profiles satisfy
\begin{equation*}
\begin{cases}
a_{i}(-\Delta u_{i} + P^\perp u_i )= \mu_{ii} u_{i} & \text{ in the open set } O_1=\{|\mathbf{u}|>0  \}\\
b_{i}(-\Delta v_{i}+Q^\perp v_i )=  \nu_{ii} v_{i}  & \text{ in the open set } O_2= \{|\mathbf{v}|>0 \};
\end{cases}
\end{equation*}
\item for any $x_0 \in \R^N$ and $r \in (0,\mathrm{dist}(x_0,\partial \Omega))$, the following identity holds
\begin{multline*}
    (2 - N) \sum_{i=1}^{k}  \int_{B_r(x_0)}  \Big(a_{i} |\nabla u_{i}|^2 +  b_{i} \Big(|\nabla v_{i}|^2  \Big)  \\
    =  \sum_{i=1}^k \int_{\partial B_r(x_0)} \Big( a_{i} r (2(\partial_\nu u_{i})^2-|\nabla u_{i}|^2)  + b_{i} r (2(\partial_\nu v_{i})^2-|\nabla v_{i}|^2) \Big)\\
 +\sum_{i=1}^k \int_{\partial B_r(x_0)} r(\mu_{ii} u_{i}^2+\nu_{i} v_{i}^2)-\sum_{i=1}^k \int_{B_r(x_0)} N(\mu_{ii} u_{i}^2 + \nu_{i} v_{i}^2)\\
 -\sum_{i=1}^k  \int_{B_r(x_0)} \Big(2 a_{i} (P^\perp u_i) \nabla u_{i}(x_0) \cdot (x-x_0) +2 b_{i} (Q^\perp v_i) \nabla v_{i}\cdot (x-x_0)\Big)
\end{multline*}
\end{enumerate}
\end{proposition}
\begin{proof}
The positivity of the coefficients in  \eqref{eq:limits_a_b} follows directly from Proposition \ref{unif h1 and linf}. Testing the equation of $u_{i,\beta}$ in \eqref{eq:equation_for_u_p_v_p} by $u_{j,\beta}$, we see that
\begin{align*}
\mu_{ij,\beta}&=\delta_{ij} a_{i,\beta} \Big(\int_\Omega |\nabla u_{i,\beta}|^2+(P^\perp u_{i,\beta})^2 \Big)  +\beta \int_\Omega u_{i,\beta}u_{j,\beta} \Big(\sum_{j=1}^k u_{j,\beta}^2\Big)^{q-1} \Big(\sum_{j=1}^k v_{j,\beta}\Big)^q\\
			&\to \delta_{ij} a_{i} \Big(\int_\Omega |\nabla u_{i}|^2+(P^\perp u_i)^2 \Big) 
\end{align*}
as $\beta\to \infty$ by \eqref{eq:interactionterm_to0}, and the same for $\nu_{ij,\beta}$. From this follows (1) and (2). As for (3), it follows exactly as in the proof of \cite[Corollary 3.16]{RamosTavaresTerracini}, taking again into account the strong $H^1_0$--convergence of minimizers (Proposition \ref{prop_limit_in_p}-(1)) and the vanishing property of the interaction term \eqref{eq:interactionterm_to0}.
\end{proof}

In order to reach the conclusion of  Theorem \ref{thm:main_result}, it is convenient to introduce the following definition. Given a measurable set $\omega\subset \R^n$, we define $\widetilde \lambda_k(\omega,\boldsymbol{\phi})$ as the $k$-eigenvalue (counting multiplicities) of the operator $-\Delta +P^\perp$ in $\widetilde H^1_0(\omega)$, which can be characterized as
\[
\widetilde \lambda_k(\omega,\boldsymbol{\phi})=\mathop{\inf_{M\subset \widetilde H^1_0(\omega)}}_{\dim M= k} \sup_{u\in M} \left( \left.\int_\omega |\nabla u|^2+ (P^\perp u )^2 \right) \right/ \int_\omega u^2.
\]
We define $\widetilde \lambda_k(\omega,\boldsymbol{\psi})$ is an analogous way. Clearly, we have
\begin{equation}\label{eq:monotonicity_eigenvalues_phipsi}
\lambda_k(\omega,\boldsymbol{\phi}), \lambda_k (\omega,\boldsymbol{\psi})\geq \lambda_k(\omega).
\end{equation}

\begin{proof}[Conclusion of the proof of Theorem \ref{thm:main_result}]
Let
\[
	\lim_{\beta\to \infty} \mathbf{u}_\beta=:\mathbf{u}=(u_1,\ldots, u_k), \qquad \lim_{\beta\to \infty} \mathbf{v}_\beta=:\mathbf{v}=(v_1,\ldots, v_k)
\]
and $(O_1,O_2):=( \{|\mathbf{u}|>0  \},\{|\mathbf{v}|>0\})$. We recall that $\mathbf{u}$ and $\mathbf{v}$ are continuous functions, thus $O_1$ and $O_2$ are open subsets of $\Omega$.
By Proposition \ref{prop_limit_in_p}--(4) and inequality \eqref{eq:monotonicity_eigenvalues_phipsi},
\[
\int_\Omega \Big( |\nabla u_i|^2 + (P^\perp u_i)^2) \geq \lambda_i(O_1,\boldsymbol{\phi})\geq \lambda_i(O_1),\ \int_\Omega \Big( |\nabla v_i|^2 + (Q^\perp v_i)^2) \geq \lambda_i(O_2,\boldsymbol{\psi})\geq \lambda_i(O_2)
\]
for every $i=1,\ldots, k$. Therefore, using the monotonicity of $F$ and $\varphi$ together with Propositions \ref{unif h1 and linf}, \ref{prop_limit_in_p} and \ref{prop:limit_in_p}, 
\begin{equation}\label{eq:chain_of_identities}
\begin{aligned}
\widetilde c  = & F( \varphi(\widetilde \lambda_1(\omega_1),\ldots, \widetilde \lambda_k(\omega_1)),  \varphi(\widetilde \lambda_1(\omega_2),\ldots, \widetilde \lambda_k(\omega_2)))\\
		=& \lim_{\beta} c_\beta  \\
		=&\lim_\beta E_\beta(\mathbf{u}_\beta,\mathbf{v}_\beta)\\
		       =& F\Big( \varphi\Big( \int_\Omega |\nabla u_1|^2 + (P^\perp u_1)^2, \ldots, \int_\Omega |\nabla u_k|^2 + (P^\perp
		       u_k)^2 \Big), \\
		       	&\phantom{F\Big( \varphi(} \varphi\Big( \int_\Omega |\nabla v_1|^2 + (Q^\perp \mathbf{v})^2_1, \ldots, \int_\Omega |\nabla v_k|^2 + (Q^\perp \mathbf{v})^2_k \Big)       \Big) \\
			\geq  & F( \varphi(\lambda_1(O_1,\boldsymbol{\phi}),\ldots, \lambda_k(O_1,\boldsymbol{\phi})),  \varphi(\lambda_1(O_2,\boldsymbol{\psi}),\ldots, \lambda_k(O_2,\boldsymbol{\psi}))) \\
			\geq  & F( \varphi(\lambda_1(O_1),\ldots, \lambda_k(O_1)),  \varphi(\lambda_1(O_2),\ldots, \lambda_k(O_2))) \\
			\geq  & F( \varphi(\widetilde \lambda_1(O_1),\ldots, \widetilde \lambda_k(O_1)),  \varphi(\widetilde \lambda_1(O_2),\ldots, \widetilde \lambda_k(O_2))) \\
			\geq & \widetilde c.
	\end{aligned}
\end{equation}
Therefore all inequalities are in fact equalities, $(O_1,O_2)$ is an (open) optimal partition for $c= \widetilde c$, and (by the strict monotonicity of $F$ and $\varphi$) $\lambda_i(O_1)=\lambda_i(O_1,\boldsymbol{\phi})$, $\lambda_i(O_2)=\lambda_i(O_2,\boldsymbol{\psi})$ for every $i=1,\ldots, k$.

We now claim that $P^\perp u_i=Q^\perp v_i=0$. Indeed, for $i=1$:
\begin{align*}
\lambda_1(O_1)=\lambda_1(O_1,\boldsymbol{\phi}) =\int_\Omega |\nabla u_1|^2+(P^\perp u_1)^2\geq \int_\Omega |\nabla u_1|^2\geq \lambda_1(O_2),
\end{align*}
so that $P^\perp u_1=0$. Moreover,
\[
\int_\Omega \nabla u_1 \cdot \nabla u_2 =\int_\Omega \nabla u_1\cdot \nabla u_2+(P^\perp u_1)(P^\perp u_2)=0,
\]
and 
\begin{align*}
\lambda_2(O_1)=\lambda_2(O_1,\boldsymbol{\phi}) =\int_\Omega |\nabla u_2|^2+(P^\perp u_2)^2\geq \int_\Omega |\nabla u_2|^2\geq \lambda_2(O_2),
\end{align*}
hence $P^\perp u_2=0$. By iterating this procedure, we obtain $P^\perp u_i$=0 for $i=1,\dots, k$ and, analogously, $Q^\perp v_i=0$, which proves our claim. 

From this we deduce that
\[
-\Delta u_i=\lambda_i(O_1) u_i \text{ in $O_1$},\qquad -\Delta v_i=\lambda_i(O_2) v_i \text{ in $O_2$}
\]
and $\lambda_i(\omega_1)=\lambda_i(O_1)$ for $i=1,\ldots, k$. Moreover  $\mathbf{u}\in L(\boldsymbol{\phi})$, $\mathbf{v}\in L(\boldsymbol{\psi})$, that is,
 \[
\mathbf{u}=M\boldsymbol{\phi},\qquad \mathbf{v}=N\boldsymbol{\psi} 
 \]
for  $M:=(\langle u_i,\phi_j\rangle_{L^2(\Omega)})_{i,j}, N_{ij}:=(\langle v_i,\psi_i\rangle_{L^2(\Omega)})_{i,j}\in \R^{k\times k}$ and, since $(\mathbf{u},\mathbf{v}), (\boldsymbol{\phi},\boldsymbol{\psi})\in \Sigma(L^2)$, then actually $M,N\in \Oeh_k(\R)$, being block diagonal  matrices:
\begin{equation}\label{eq:blockdiagonalmatrix}
M=\diag(M_{1},\ldots, M_{l_1}),\quad  N=\diag(N_{1},\ldots, N_{l_2}),
\end{equation}
where the dimension of each block is at most equal to the dimension of the eigenspace of the associated eigenvalue, and each block is itself an orthogonal matrix.

 This has many important consequences:
\begin{enumerate}
\item In the local Pohozaev identities of Proposition \ref{prop:limit_in_p}-(3) we have $P^\perp u_i=Q^\perp v_i=0$, which corresponds to the statement in \cite[Corollary 3.16]{RamosTavaresTerracini}. Therefore we are in the exact framework of Sections 3 and 4 of \cite{RamosTavaresTerracini}, which implies by Theorem 2.2 therein that $u_i,v_i$ are Lipschitz continuous, $(O_1,O_2)$ is a regular partition, and, given $x_0$ in the regular part of the free boundary,
\begin{equation*}
\mathop{\lim_{x\to x_0}}_{x\in  O_1} \sum_{j=1}^{k} a_{j} |\nabla  u_j(x)|^2=\mathop{\lim_{x\to x_0}}_{x\in  O_2} \sum_{j=1}^{k} b_j |\nabla  v_j(x)|^2\neq 0,
\end{equation*}
where
\begin{align*}
&a_i=\partial F_1(\varphi(\lambda_1(\omega_1),\ldots, \lambda_k(\omega_1)),\varphi(\lambda_1(\omega_2),\ldots, \lambda_k(\omega_2)))\partial_i \varphi(\lambda_1(\omega_1),\ldots, \lambda_k(\omega_1)),\\
&b_i=\partial F_2(\varphi(\lambda_1(\omega_1),\ldots, \lambda_k(\omega_1)),\varphi(\lambda_1(\omega_2),\ldots, \lambda_k(\omega_2)))\partial_i \varphi(\lambda_1(\omega_2),\ldots, \lambda_k(\omega_2))
\end{align*}
Since $\varphi$ is symmetric, then $a_i=a_j$ whenever $\lambda_i(\omega_1)=\lambda_j(\omega_1)$, and the same holds true for the coefficients $b_i$. Combining this remark with the orthogonality of the block matrices in \eqref{eq:blockdiagonalmatrix}, we deduce that also
\begin{equation}\label{eq:reflectionlaw}
\mathop{\lim_{x\to x_0}}_{x\in  O_1} \sum_{j=1}^{k} a_{j} |\nabla  \phi_j(x)|^2=\mathop{\lim_{x\to x_0}}_{x\in  O_2} \sum_{j=1}^{k} b_j |\nabla  \psi_j(x)|^2\neq 0.
\end{equation}
Moreover we find that \eqref{eq:reflectionlaw} does not depend on the starting configuration  $\boldsymbol{\varphi}, \boldsymbol{\psi}$.
\item  Since $M$ and $N$ are invertible, $\phi=M^{-1} \mathbf{u}$ and $\psi=N^{-1} \mathbf{v}$ a.e. in $\Omega$, and since $\mathbf{u},\mathbf{v}$ are Lipschitz continuous, then each $\phi_i$ and $\psi_i$ has a Lipschitz continuous representative.
\item For a.e. $x\in \Omega$ we find
\[
    |\mathbf{u}|^2(x) = \mathbf{u}(x) \cdot \mathbf{u}(x) = M \boldsymbol{\phi}(x) \cdot M \boldsymbol{\phi}(x) = |\boldsymbol{\phi}|^2(x),\qquad |\mathbf{v}|^2(x)=|\boldsymbol{\psi}|^2(x).
\]
\end{enumerate}
Therefore we have $O_j \subseteq \omega_j$ up to a set of Lebesgue measure zero, $\lambda_i(O_j)=\widetilde \lambda_i(O_j) \geq \widetilde \lambda_i(\omega_j)$ for $j=1,2$, $i=1,\ldots, k$. Combining this with the strict monotonicity of $F$ and $\varphi$ and \eqref{eq:chain_of_identities}, we obtain the equality between the eigenvalues. Moreover, the regularity results of $(O_1,O_2)$ allow to conclude that $|O_i\triangle \omega_i |=0$.

We are left to show the spectral gap property, that is, to prove that $\widetilde \lambda_k(\omega_1) < \widetilde \lambda_{k+1}(\omega_1)$. For this purpose, let $E \subset \widetilde H^1_0(\omega_1)$ be the (generalized) eigenspace associated to the eigenvalue $\widetilde \lambda_k(\omega_1)$ and let $\ell \in \N$ be the number of eigenvalue of $\omega_1$ that are strictly less than $\widetilde \lambda_k(\omega_1)$. Our goal is to show that
\[
    \ell + \mathrm{dim}(E) = k.
\]
Assume, in view of a contradiction, that $\widetilde \lambda_k(\omega_1)=\widetilde \lambda_{k+1}(\omega_1)$ or, more generally, that
\begin{equation}\label{eqn contr}
    \ell + \mathrm{dim}(E) \geq k + 1.
\end{equation}
To start off, we apply the previous reasoning to any vector $\boldsymbol{\phi} = (\phi_1, \dots, \phi_{\ell}, \bar \phi_{\ell+1}, \dots, \bar \phi_{k})$ where $\bar \phi_{\ell+1}, \dots, \bar \phi_{k}$ are $k-\ell$ orthonormal functions in $E$. This shows that all the eigenfunctions in $E$ have a Lipschitz representative and that $E$ is made of standard eigenfunctions. In particular, by \eqref{eq:reflectionlaw}, replacing one eigenfunction at the time, for any $\phi_i \perp \phi_j$ in any orthonormal base of $E$ we deduce
\begin{equation}\label{eqn stronger refl}
    |\nabla \phi_{i}|^2=|\nabla \phi_{j}|^2 
\end{equation}
on the regular part of the free boundary. Let now $S\subset \overline{\Omega}$ stand for the support of $E$
\[
    S = \supp \left(\sum_{i=1}^{\mathrm{dim}(E)} |\phi_i| \right) = \mathrm{clo}\left(\sum_{i=1}^{\mathrm{dim}(E)} |\phi_i| >  0\right).
\]
We claim that, under \eqref{eqn contr}, $S$ has a unique connected component. Assume the opposite  and pick two normalized functions $\phi', \phi'' \in E$ with disjoint supports (this is possible since $S$ is disconnected, and $\phi', \phi''$ are orthonormal by construction), and consider other $\mathrm{dim}(E)-2$ functions to complete an orthonormal base of $E$. We immediately find a contradiction with \eqref{eqn stronger refl}. Hence, up to a change of sign, letting $w:= \phi_{i}-\phi_{j}$ for any $\phi_i \perp \phi_j$ in any orthonormal base of $E$, we find
\[
    \begin{cases}
        -\Delta w=\widetilde \lambda_k(O_1) w &\text{ in $O_1$}\\
        w = |\nabla w|=0 &\text{ on $\partial O_1$}.
    \end{cases}
\]
But then, by Hopf's lemma, we have $w=0$ that is $\phi_i = \phi_j$, a contradiction. The same reasoning holds true for $\widetilde \lambda_k(\omega_2)$.
\end{proof}

\subsection*{Acknowledgements}

H. Tavares was supported by the Portuguese government through FCT -- Funda\c c\~ao para a Ci\^encia e a Tecnologia, I.P., under the projects PTDC/MAT-PUR/28686/2017, UID/MAT/04561/2013UID/MAT/04561/2013 and UIDB/MAT/04459/2020. H. Tavares would also like to acknowledge the Faculty of Sciences of the University of Lisbon for granting a semestral sabbatical leave, during which part of this work was developed.

This work was partially supported by the project ANR-18-CE40-0013 SHAPO financed by the French Agence Nationale de la Recherche (ANR).


\end{document}